\documentclass[11pt, a4paper]{article}

\usepackage[a4paper]{geometry}

\usepackage{amsmath}
\usepackage{amssymb}
\usepackage{amsthm}
\usepackage{bbm}
\usepackage{mathtools} 

\usepackage[utf8]{inputenc}
\usepackage[english]{babel}

\usepackage[T1]{fontenc}

\usepackage{authblk} 

\usepackage{xparse} \usepackage{xifthen} \usepackage{etoolbox} \usepackage{chngcntr} \usepackage{xspace}
\usepackage{nicefrac} \usepackage{mleftright} \mleftright 

\usepackage[bookmarksnumbered,bookmarksopen,unicode,hidelinks]{hyperref}
\usepackage[capitalise,noabbrev]{cleveref} 

\RequirePackage{doi}
\usepackage{orcidlink}

\usepackage[disable]{todonotes}

\newboolean{colorcommands}
\setboolean{colorcommands}{false}

\theoremstyle{plain} \newtheorem{theorem}{Theorem}[section]
\newtheorem{corollary}[theorem]{Corollary}
\newtheorem{remark}[theorem]{Remark}
\newtheorem{proposition}[theorem]{Proposition}
\newtheorem{lemma}[theorem]{Lemma}
\newtheorem{claim}[theorem]{Claim}

\newcounter{previoustheorem}
\makeatletter
\newenvironment{nestedenvironment}[1]{\counterwithin*{theorem}{previoustheorem}
    \def\nestedenvironmentcounter{#1}\protected@edef\theprevioustheorem{\csname the#1\endcsname}\setcounter{previoustheorem}{\value{#1}}\setcounter{#1}{0}\expandafter\def\csname the#1\endcsname{\theprevioustheorem\alph{#1}}\ignorespaces
}{\setcounter{\nestedenvironmentcounter}{\value{previoustheorem}}\counterwithout*{theorem}{previoustheorem} \ignorespacesafterend
}
\makeatother

\theoremstyle{definition} \newtheorem{definition}[theorem]{Definition}
\newtheorem{example}[theorem]{Example}

\DeclareMathOperator{\match}{M}
\DeclareMathOperator{\perfmatch}{PM}
\DeclareMathOperator{\linspan}{span}
\DeclareMathOperator{\flowext}{FLX}

\ExplSyntaxOn
\DeclareExpandableDocumentCommand{\IfNoValueOrEmptyTF}{mmm}
{
  \IfNoValueTF{#1}{#2}
  {
    \tl_if_empty:nTF {#1} {#2} {#3}
  }
}
\ExplSyntaxOff

\ifthenelse{\boolean{colorcommands}}{
  \NewDocumentCommand\ie{}{\textcolor{red!50}{i.\,e.}}

}{
  \NewDocumentCommand\ie{}{i.\,e.}

}

\ifthenelse{\boolean{colorcommands}}{
  \NewDocumentCommand\CTPtext{}{\textcolor{red!50}{cyclic transversal polytope}\xspace}
}{
  \NewDocumentCommand\CTPtext{}{cyclic transversal polytope\xspace}
}

\newcommand{\R}{\mathbb{R}}

\newcommand{\Z}{\mathbb{Z}}
\newcommand{\F}{\mathbb{F}_2}
\newcommand{\B}{\mathcal{B}} 

\ifthenelse{\boolean{colorcommands}}{
  \newcommand{\zerovec}{\textcolor{red!50}{\mathbb{O}}}
  \newcommand{\onevec}{\textcolor{red!50}{\mathbbm{1}}}
  \newcommand{\unitvec}{\textcolor{red!50}{\mathbbm{e}}}
}{
  \newcommand{\zerovec}{\mathbb{O}}
  \newcommand{\onevec}{\mathbbm{1}}
  \newcommand{\unitvec}{\mathbbm{e}}
}

\ifthenelse{\boolean{colorcommands}}{
  \DeclareMathOperator{\losOP}{\textcolor{red!50}{los}}
}{
  \DeclareMathOperator{\losOP}{los}
}
\newcommand{\los}[3]{\losOP^{#1}_{{#2},{#3}}}

\ifthenelse{\boolean{colorcommands}}{
  
  \DeclareMathOperator{\conv}{\textcolor{red!50}{conv}}
  
  \DeclareMathOperator{\bigO}{\textcolor{red!50}{\mathcal{O}}}
  \DeclareMathOperator{\len}{\textcolor{red!50}{len}}
  \DeclareMathOperator{\rank}{\textcolor{red!50}{rank}}
  \DeclareMathOperator{\size}{\textcolor{red!50}{size}}

  \DeclareMathOperator{\true}{\textcolor{red!50}{true}}
  \DeclareMathOperator{\false}{\textcolor{red!50}{false}}
}{
  
  \DeclareMathOperator{\conv}{conv}
  
  \DeclareMathOperator{\bigO}{\mathcal{O}}
  \DeclareMathOperator{\len}{len}
  \DeclareMathOperator{\rank}{rank}
  \DeclareMathOperator{\size}{size}

  \DeclareMathOperator{\true}{true}
  \DeclareMathOperator{\false}{false}
}

\ifthenelse{\boolean{colorcommands}}{
  \DeclareMathOperator{\NP}{\textcolor{red!50}{NP}} }{
  \DeclareMathOperator{\NP}{NP} }

\NewDocumentCommand\blockconfigglyph{}{\B}
\NewDocumentCommand\blockglyph{}{\B}
\NewDocumentCommand\blockelementglyph{}{\xi}

\ifthenelse{\boolean{colorcommands}}{
  \NewDocumentCommand\blockconfig{oo}{\textcolor{red!50}{\IfNoValueOrEmptyTF{#2}{\IfNoValueOrEmptyTF{#1}{\blockconfigglyph}{\blockconfigglyph_{#1}}}{\IfNoValueOrEmptyTF{#1}{\blockconfigglyph^{#2}}{\blockconfigglyph_{#1}^{#2}}}}} \NewDocumentCommand\block{oo}{\textcolor{red!50}{\IfNoValueOrEmptyTF{#2}{\IfNoValueOrEmptyTF{#1}{\blockglyph}{\blockglyph({#1})}}{\IfNoValueOrEmptyTF{#1}{\blockglyph_{#2}}{\blockglyph_{#2}({#1})}}}} \NewDocumentCommand\blockelement{soo}{\textcolor{red!50}{\IfNoValueOrEmptyTF{#3}{\blockelementglyph\IfNoValueOrEmptyTF{#2}{}{\IfBooleanTF{#1}{\left[#2\right]}{\left(#2\right)}}}{\blockelementglyph_{#3}\IfNoValueOrEmptyTF{#2}{}{\IfBooleanTF{#1}{\left[#2\right]}{\left(#2\right)}}}}} \NewDocumentCommand\blockfamily{D<>{1}O{\ctorder}}{\textcolor{red!50}{(\block[#1], \ldots, \block[#2])}}

  \NewDocumentCommand\vectorsection{mo}{\textcolor{red!50}{\IfNoValueOrEmptyTF{#2}{#1\vert}{#1\vert_{#2}}}} 

  \NewDocumentCommand\ctdim{o}{\textcolor{red!50}{\IfNoValueOrEmptyTF{#1}{r}{r_{#1}}}} \NewDocumentCommand\ctrank{o}{\textcolor{red!50}{\IfNoValueOrEmptyTF{#1}{d}{d_{#1}}}} \NewDocumentCommand\ctorder{o}{\textcolor{red!50}{\IfNoValueOrEmptyTF{#1}{n}{n_{#1}}}} \NewDocumentCommand\ctsize{o}{\textcolor{red!50}{\IfNoValueOrEmptyTF{#1}{s}{s_{#1}}}} \NewDocumentCommand\venuespace{o}{\textcolor{red!50}{\IfNoValueTF{#1}{\F^{\ctrank}}{\F^{#1}}}} \NewDocumentCommand\linsubspace{o}{\textcolor{red!50}{\IfNoValueTF{#1}{L^{\blockconfig}}{L^{#1}}}} }{
  \NewDocumentCommand\blockconfig{oo}{\IfNoValueOrEmptyTF{#2}{\IfNoValueOrEmptyTF{#1}{\blockconfigglyph}{\blockconfigglyph_{#1}}}{\IfNoValueOrEmptyTF{#1}{\blockconfigglyph^{#2}}{\blockconfigglyph_{#1}^{#2}}}} \NewDocumentCommand\block{oo}{\IfNoValueOrEmptyTF{#2}{\IfNoValueOrEmptyTF{#1}{\blockglyph}{\blockglyph({#1})}}{\IfNoValueOrEmptyTF{#1}{\blockglyph_{#2}}{\blockglyph_{#2}({#1})}}} \NewDocumentCommand\blockelement{soo}{\IfNoValueOrEmptyTF{#3}{\blockelementglyph\IfNoValueOrEmptyTF{#2}{}{\IfBooleanTF{#1}{\left[#2\right]}{\left(#2\right)}}}{\blockelementglyph_{#3}\IfNoValueOrEmptyTF{#2}{}{\IfBooleanTF{#1}{\left[#2\right]}{\left(#2\right)}}}} \NewDocumentCommand\blockfamily{D<>{1}O{\ctorder}}{(\block[#1], \ldots, \block[#2])}

  \NewDocumentCommand\vectorsection{mo}{\IfNoValueOrEmptyTF{#2}{#1\vert}{#1\vert_{#2}}} 

  \NewDocumentCommand\ctdim{o}{\IfNoValueOrEmptyTF{#1}{r}{r_{#1}}} \NewDocumentCommand\ctrank{o}{\IfNoValueOrEmptyTF{#1}{d}{d_{#1}}} \NewDocumentCommand\ctorder{o}{\IfNoValueOrEmptyTF{#1}{n}{n_{#1}}} \NewDocumentCommand\ctsize{o}{\IfNoValueOrEmptyTF{#1}{s}{s_{#1}}} \NewDocumentCommand\venuespace{o}{\IfNoValueTF{#1}{\F^{\ctrank}}{\F^{#1}}} \NewDocumentCommand\linsubspace{o}{\IfNoValueTF{#1}{L^{\blockconfig}}{L^{#1}}} }

 \NewDocumentCommand\norm{mo}{\IfNoValueOrEmptyTF{#2}{\left\lVert #1 \right\rVert}{\left\lVert #1 \right\rVert_{#2}}}  

\ifthenelse{\boolean{colorcommands}}{
  \NewDocumentCommand\CT{oo}{\textcolor{red!50}{\IfValueTF{#2}{\operatorname{CT}_{#2}}{\operatorname{CT}}\IfNoValueOrEmptyTF{#1}{}{\left(#1\right)}}} \NewDocumentCommand\TP{oo}{\textcolor{red!50}{\IfValueTF{#2}{\operatorname{TP}_{#2}}{\operatorname{TP}}\IfNoValueOrEmptyTF{#1}{}{\left(#1\right)}}} \NewDocumentCommand\CTP{oo}{\textcolor{red!50}{\IfValueTF{#2}{\operatorname{CTP}_{#2}}{\operatorname{CTP}}\IfNoValueOrEmptyTF{#1}{}{\left(#1\right)}}} \NewDocumentCommand\BSP{o}{\textcolor{red!50}{\operatorname{BSP}\IfNoValueOrEmptyTF{#1}{}{\left(#1\right)}}} \NewDocumentCommand\EVEN{o}{\textcolor{red!50}{\operatorname{EVEN}\IfNoValueOrEmptyTF{#1}{}{\left(#1\right)}}} \NewDocumentCommand\PIP{o}{\textcolor{red!50}{\operatorname{CTP}\IfNoValueOrEmptyTF{#1}{}{\left(#1\right)}}} \NewDocumentCommand\STAB{o}{\textcolor{red!50}{\operatorname{STAB}\IfNoValueOrEmptyTF{#1}{}{\left(#1\right)}}} \NewDocumentCommand\PACK{o}{\textcolor{red!50}{\operatorname{PACK}\IfNoValueOrEmptyTF{#1}{}{\left(#1\right)}}} \NewDocumentCommand\CROSS{o}{\textcolor{red!50}{\operatorname{CROSS}\IfNoValueOrEmptyTF{#1}{}{\left(#1\right)}}} \NewDocumentCommand\TSP{o}{\textcolor{red!50}{\operatorname{TSP}\IfNoValueOrEmptyTF{#1}{}{\left(#1\right)}}} \NewDocumentCommand\STP{o}{\textcolor{red!50}{\operatorname{STP}\IfNoValueOrEmptyTF{#1}{}{\left(#1\right)}}} }{
  \NewDocumentCommand\CT{oo}{\IfValueTF{#2}{\operatorname{CT}_{#2}}{\operatorname{CT}}\IfNoValueOrEmptyTF{#1}{}{\left(#1\right)}} \NewDocumentCommand\TP{oo}{\IfValueTF{#2}{\operatorname{TP}_{#2}}{\operatorname{TP}}\IfNoValueOrEmptyTF{#1}{}{\left(#1\right)}} \NewDocumentCommand\CTP{oo}{\IfValueTF{#2}{\operatorname{CTP}_{#2}}{\operatorname{CTP}}\IfNoValueOrEmptyTF{#1}{}{\left(#1\right)}} \NewDocumentCommand\BSP{o}{\operatorname{BSP}\IfNoValueOrEmptyTF{#1}{}{\left(#1\right)}} \NewDocumentCommand\EVEN{o}{\operatorname{EVEN}\IfNoValueOrEmptyTF{#1}{}{\left(#1\right)}} \NewDocumentCommand\PIP{o}{\operatorname{CTP}\IfNoValueOrEmptyTF{#1}{}{\left(#1\right)}} \NewDocumentCommand\STAB{o}{\operatorname{STAB}\IfNoValueOrEmptyTF{#1}{}{\left(#1\right)}} \NewDocumentCommand\PACK{o}{\operatorname{PACK}\IfNoValueOrEmptyTF{#1}{}{\left(#1\right)}} \NewDocumentCommand\CROSS{o}{\operatorname{CROSS}\IfNoValueOrEmptyTF{#1}{}{\left(#1\right)}} \NewDocumentCommand\TSP{o}{\operatorname{TSP}\IfNoValueOrEmptyTF{#1}{}{\left(#1\right)}} \NewDocumentCommand\STP{o}{\operatorname{STP}\IfNoValueOrEmptyTF{#1}{}{\left(#1\right)}} }

\ifthenelse{\boolean{colorcommands}}{
  \NewDocumentCommand\simplex{o}{\textcolor{red!50}{\IfNoValueOrEmptyTF{#1}{\Delta}{\Delta_{#1}}}}
  \NewDocumentCommand\zeroone{o}{\textcolor{red!50}{\IfNoValueOrEmptyTF{#1}{\setdef{0,1}}{\setdef{0,1}^{#1}}}}
  \NewDocumentCommand\unitint{o}{\textcolor{red!50}{\IfNoValueOrEmptyTF{#1}{\left[0,1\right]}{\left[0,1\right]^{#1}}}}

  \NewDocumentCommand\affspace{o}{\textcolor{red!50}{\mathbb{A}\IfNoValueOrEmptyTF{#1}{}{\left(#1\right)}}} \NewDocumentCommand\cube{o}{\textcolor{red!50}{\operatorname{C}\IfNoValueOrEmptyTF{#1}{}{\left(#1\right)}}}

  \NewDocumentCommand\relproj{oo}{\textcolor{red!50}{\IfValueTF{#2}{f_{#2}}{f}\IfNoValueOrEmptyTF{#1}{}{\left(#1\right)}}}
  \NewDocumentCommand\relmap{oo}{\textcolor{red!50}{\IfValueTF{#1}{\kappa_{#1}}{\kappa}\IfValueTF{#2}{\left(#2\right)}{}}}
  \NewDocumentCommand\projrelax{oo}{\textcolor{red!50}{\IfNoValueOrEmptyTF{#2}{\operatorname{RP}}{\operatorname{RP}_{#2}}}\IfNoValueOrEmptyTF{#1}{}{\left(#1\right)}} \NewDocumentCommand\projrelaxstrengthened{oo}{\textcolor{red!50}{\IfNoValueOrEmptyTF{#2}{\operatorname{RP}^+}{\operatorname{RP}^+_{#2}}\IfNoValueOrEmptyTF{#1}{}{\left(#1\right)}}} \NewDocumentCommand\netproj{mo}{\textcolor{red!50}{\pi_{#1}\IfNoValueOrEmptyTF{#2}{}{\left(#2\right)}}} }{
  \NewDocumentCommand\simplex{o}{\IfNoValueOrEmptyTF{#1}{\Delta}{\Delta_{#1}}}
  \NewDocumentCommand\zeroone{o}{\IfNoValueOrEmptyTF{#1}{\setdef{0,1}}{\setdef{0,1}^{#1}}}
  \NewDocumentCommand\unitint{o}{\IfNoValueOrEmptyTF{#1}{\left[0,1\right]}{\left[0,1\right]^{#1}}}

  \NewDocumentCommand\affspace{o}{\mathbb{A}\IfNoValueOrEmptyTF{#1}{}{\left(#1\right)}} \NewDocumentCommand\cube{o}{\operatorname{C}\IfNoValueOrEmptyTF{#1}{}{\left(#1\right)}}

  \NewDocumentCommand\relproj{oo}{\IfValueTF{#2}{f_{#2}}{f}\IfNoValueOrEmptyTF{#1}{}{\left(#1\right)}}
  \NewDocumentCommand\relmap{oo}{\IfValueTF{#1}{\kappa_{#1}}{\kappa}\IfValueTF{#2}{\left(#2\right)}{}}
  \NewDocumentCommand\projrelax{oo}{\IfNoValueOrEmptyTF{#2}{\operatorname{RP}}{\operatorname{RP}_{#2}}\IfNoValueOrEmptyTF{#1}{}{\left(#1\right)}} \NewDocumentCommand\projrelaxstrengthened{oo}{\IfNoValueOrEmptyTF{#2}{\operatorname{RP}^+}{\operatorname{RP}^+_{#2}}\IfNoValueOrEmptyTF{#1}{}{\left(#1\right)}} \NewDocumentCommand\netproj{mo}{\pi_{#1}\IfNoValueOrEmptyTF{#2}{}{\left(#2\right)}} }

\NewDocumentCommand\overlinemod{m}{\bar{#1}} 

\NewDocumentCommand\venuespacemod{o}{\IfNoValueTF{#1}{\F^{\overlinemod{\ctrank}}}{\F^{\overlinemod{#1}}}}
\NewDocumentCommand\blockconfigmod{oo}{\IfNoValueOrEmptyTF{#2}{\IfNoValueOrEmptyTF{#1}{\overlinemod{\blockconfigglyph}}{\overlinemod{\blockconfigglyph}_{#1}}}{\IfNoValueOrEmptyTF{#1}{\overlinemod{\blockconfigglyph}^{#2}}{\overlinemod{\blockconfigglyph}_{#1}^{#2}}}}
\NewDocumentCommand\blockmod{oo}{\IfNoValueOrEmptyTF{#2}{\IfNoValueOrEmptyTF{#1}{\overlinemod{\blockglyph}}{\overlinemod{\blockglyph}_{#1}}}{\IfNoValueOrEmptyTF{#1}{\overlinemod{\blockglyph}^{#2}}{\overlinemod{\blockglyph}_{#1}^{#2}}}}
\NewDocumentCommand\blockelementmod{oo}{\IfNoValueOrEmptyTF{#2}{\overlinemod{\blockelementglyph}\IfNoValueOrEmptyTF{#1}{}{\left(#1\right)}}{\overlinemod{\blockelementglyph}_{#2}\IfNoValueOrEmptyTF{#1}{}{\left(#1\right)}}}
\NewDocumentCommand\blockfamilymod{D<>{1}O{\ctorder}}{\blockmod[#1], \ldots, \blockmod[#2]}

\ifthenelse{\boolean{colorcommands}}{
  \NewDocumentCommand\suchthat{}{\textcolor{red!50}{\;\ifnum\currentgrouptype=16 \middle\fi|\;}} \NewDocumentCommand\setdef{mo}{\textcolor{red!50}{\left\{#1\IfNoValueOrEmptyTF{#2}{}{ \suchthat #2}\right\}}} \NewDocumentCommand\tuple{mo}{\textcolor{red!50}{\left(#1\IfNoValueOrEmptyTF{#2}{}{ \suchthat #2}\right)}} }{
  \NewDocumentCommand\suchthat{}{\;\ifnum\currentgrouptype=16 \middle\fi|\;} \NewDocumentCommand\setdef{mo}{\left\{#1\IfNoValueOrEmptyTF{#2}{}{ \suchthat #2}\right\}} \NewDocumentCommand\tuple{mo}{\left(#1\IfNoValueOrEmptyTF{#2}{}{ \suchthat #2}\right)} }

\title{Binary Cyclic Transversal Polytopes}
\author[1]{Jonas Frede~\orcidlink{0009-0004-3806-7726}}
\author[1]{Volker Kaibel~\orcidlink{0000-0002-0388-7597}}
\author[2]{Maximilian Merkert~\orcidlink{0000-0002-7838-445X}}
\date{\phantom{\today}}

\affil[1]{
Otto von Guericke University Magdeburg, Department of Mathematics, Universit\"{a}tsplatz 2, 39106 Magdeburg, Germany
}

\affil[2]{
Technische Universit\"{a}t Braunschweig, Institute for Mathematical Optimization, Universit\"{a}tsplatz~2, 38106 Braunschweig, Germany
}

\begin{document}

    \maketitle

    \begin{abstract}
        With every family of finitely many subsets of a finite-dimensional vector space over the Galois-field with two elements we associate a \emph{cyclic transversal polytope}. It turns out that those polytopes generalize several well-known polytopes that are relevant in combinatorial optimization, among them cut polytopes as well as stable set and matching polytopes. We introduce the class of lifted odd-set inequalities and prove results demonstrating their strength. In particular, we show that they suffice to describe cyclic transversal polytopes if the union of the sets in the family has rank at most two. We also describe extended formulations for cyclic transversal polytopes and introduce a special relaxation hierarchy for them.
\end{abstract}

    \section{Introduction}

This paper introduces binary cyclic transversal polytopes. As all the concepts we are going to define refer to the field $\F$, the binary Galois-field with two elements, we will omit the term \emph{binary} from the corresponding notions in the following.

\begin{definition}
  A \emph{block configuration} in $\F^d$ (with $d \ge 1$) is a sequence $\blockconfig = \blockfamily$ (with $n \ge 2$) of non-empty \emph{blocks} $\varnothing \ne \block[i] \subseteq \F^d$.
  We call $n$ its \emph{length}, and $\sum_{i=1}^n |\block[i]|$ its \emph{size}. We define $\linspan(\B)$ to be the linear subspace of $\F^d$ spanned by $\block[1] \cup \ldots \cup \block[n]$ and refer to $\rank(\B)\coloneqq \dim(\linspan(\B))$ as the \emph{rank} of $\blockconfig$.
  Any sequence $\xi = (\xi(1),\dots,\xi(n))$ with $\xi(i) \in \block[i]$ for all $i \in [n]$ is called a \emph{transversal} of $\blockconfig$.
  A sequence $\sigma=(\sigma(1),\dots,\sigma(n))$ with $\sigma(i) \in \F^d$ for each $i \in [n]$ is called \emph{cyclic} if $\sum_{i=1}^n \sigma(i) = \zerovec$ holds.
  A \emph{cyclic transversal} of a block configuration $\blockconfig$ is a transversal of $\blockconfig$ that is cyclic. We denote the set of all cyclic transversals of $\blockconfig$ by $\CT[\B]$.
  The \emph{incidence vector} of a transversal $\blockelement$ is the vector $x = (x^1, \ldots, x^n)$ with $x^i \in \{0,1\}^{\block[i]}$ for all $i \in [n]$ and $x^i_{\omega} = 1$ if and only if
  $\blockelement[i] = \omega$ for all $\omega\in\block[i]$.
  We refer to the convex hulls
  \[
    \CTP[\B] \subseteq \TP[\B] \subseteq \R^{\block[1]\times\cdots\times\block[n]} \eqqcolon \R^{\B}
  \]
  of the incidence vectors of all cyclic transversals of $\blockconfig$ and of all transversals of $\blockconfig$
  as the \emph{cyclic transversal polytope}
  and the \emph{transversal polytope}, respectively, associated with $\blockconfig$.
\end{definition}

While cyclic transversal polytopes are complex objects in general, transversal polytopes are rather trivial ones.

\begin{remark}
  The transversal polytope $\TP[\B]$ of a block configuration $\blockconfig$ of length $n$ is described by the nonnegativity constraints and the \emph{transversal equations}
  \begin{equation}\label{eq:transversal}
    \sum_{\omega\in\block[i]} x^i_{\omega} = 1
    \quad\text{for all }i \in [n]\,.
  \end{equation}
\end{remark}

Before we elaborate in more detail on the expressive power of the concept of cyclic transversal polytopes in Section~\ref{sec:expressive_power}, we discuss a first class of examples.

\begin{example}
  \label{ex:binary_matroids}
  For a matrix $M \in \F^{d \times n}$ we define  the \emph{binary subspace polytope}
  \[
    \BSP[M] \coloneqq \conv \{ x \in \{0,1\}^n : Mx = \zerovec \in \F^d\}\,,
  \]
  \ie, the convex hull of the kernel of $M$ viewed as a subset of $\{0,1\}^n$ (it has  also been referred to as the \emph{cycle polytope of the binary matroid} associated with $M$, see, e.g., \cite{BarahonaGroetschel86}). The most simple special case is the \emph{parity polytope}, i.e., the convex hull
  \(
    \EVEN[n] \coloneqq \BSP[[1,\dots,1]]
  \)
  of the 0/1-vectors of length $n$ with an even number of one-entries.
  Particularly interesting examples arise when the rows of $M$
  generate the $\F$-cycle space  of a graph $G$ which results in $\BSP[M]$ being affinely isomorphic to the cut-polytope of $G$.

  If $M$ does not have any column equal to $\zerovec$, then $\BSP[M]$ is affinely isomorphic to $\CTP[\B]$ with
  \(
    \B = (\{M_{\star,1},\zerovec\},\dots,\{M_{\star,q},\zerovec\},\{b\})
  \)
  (where $M_{\star,i}$ denotes the $i$-th column of $M$).

\end{example}

We will demonstrate in Section~\ref{sec:expressive_power} that also set packing polytopes, in particular stable set and matching polytopes, as well as 2-SAT-polytopes are examples of cyclic transversal polytopes. Moreover, it will turn out that each 0/1-polytope that is associated with a well-posed combinatorial optimization problem (in the sense that it is possible to efficiently verify feasibility of solutions) can be represented as a linear projection of a cyclic transversal polytope associated with a block configuration whose size is bounded polynomially in the encoding length of the instance.

As cyclic transversal polytopes thus seem to form a rather general class that comprises a rather broad spectrum of polytopes relevant in combinatorial optimization, the question arises whether the
class has enough structure to allow to identify interesting common properties of its members. The aim of this paper is to show that this seems to be the case. Our main contributions in this direction are to introduce the class of \emph{lifted odd-set inequalities} and to prove that they are facet defining for \emph{full cyclic transversal polytopes} (Section~\ref{sec:eq_and_ieneq}) as well as to show that they provide us with complete linear descriptions for cyclic transversal polytopes associated with block configurations of rank at most two (Section~\ref{sec:lowrank}). We also elaborate on extended formulations for cyclic transversal polytopes (Section~\ref{sec:ext_form}) and describe a path to form a special relaxation hierarchy for them (Section~\ref{sec:rank_relax}).

Most of the results presented in this paper as well as some further findings are contained in the dissertation~\cite{Frede2023} of one of the authors (Jonas Frede).
     \section{Fundamental properties of CTPs}
\label{sec:concept}

We start by providing a few more very basic examples of cyclic transversal polytopes.

\begin{example}
  \label{ex:simplex}
  For every subset $\varnothing \ne \Omega \subseteq \F^d$ the polytope $\CTP[\B]$ with $\B = (\Omega,\Omega)$ is affinely isomorphic to the $(|\Omega|-1)$-dimensional simplex with $|\Omega|$ vertices. Hence each $q$-dimensional simplex is affinely isomorphic to a cyclic trans\-versal polytope associated with a block configuration of size $2(q+1)$ and rank $\lceil\log_2 (q+1) \rceil$.
  Conversely, if $\B=(\block[1],\block[2])$ is a block configuration of length two, then $\CTP[\B]$ is isomorphic to a $(|\block[1]\cap\block[2]|-1)$-dimensional simplex (where the empty set is considered to be a simplex of dimension $-1$).
\end{example}

\begin{remark}
  Example~\ref{ex:simplex} in particular shows that cyclic transversal polytopes associated with block configurations of length at most two are affinely isomorphic to simplices.
\end{remark}

\begin{example}
  Let $\B_1$ and $\B_2$ be two block configurations in $\F^{d_1}$ and $\F^{d_2}$ with lengths $n_1$ and $n_2$, sizes $s_1$ and $s_2$, and ranks $r_1$ and $r_2$, respectively. We denote by $\blockconfig$ the block configuration in $\F^{d_1+d_2}$ with length $n_1+n_2$, size $s_1+s_2$, and rank $r_1+r_2$ formed by the blocks
  \(
    \block[1][1] \times \{\zerovec_{d_2}\}
    ,\dots,
    \block[n_1][1] \times \{\zerovec_{d_2}\}
  \)
  and
  \(
    \{\zerovec_{d_1}\} \times \block[1][2],
    \dots,
    \{\zerovec_{d_1}\} \times \block[n_2][2]
  \).
  Then $\CTP[\B]$ is affinely isomorphic to the Cartesian product
  \(
      \CTP[\B_1]\times\CTP[\B_2]
  \).
\end{example}

\begin{example}
  From the previous two examples one readily deduces that every $q$-dimensional cube (i.e., a Cartesian product of $q$ one-dimensional simplices) is affinely isomorphic to a cyclic transversal polytope associated with a block configuration of size $4q$ and rank $q$.
\end{example}

\begin{definition}\label{def:full_CTP}
    We define $\B(d,n)$ to be the block configuration of length $n$ that has each block equal to $\F^d$, denote $\CT[d,n] \coloneqq \CT[\B(d,n)]$, and call
    \[
      \CTP[d,n] \coloneqq \CTP[\B(d,n)]
    \]
    the \emph{full cyclic transversal polytope} of length $n$ and rank $d$.
\end{definition}

    As a first example we observe that $\CTP[1,n]$ is isomorphic to $\EVEN[n]$
    via the coordinate projection with $x_i = x^i_1$ for all $i \in [n]$ (recall that the transversal equations $x^i_0+x^i_1=1$ hold for $\CTP[1,n]$).

\begin{definition}
For any block configuration $\blockconfig$ of length $n$ in $\F^d$ we define the linear subspace
  \[
  \linsubspace[\B] :=
    \{
      x \in \R^{\B(d,n)} : x^i_{\omega} = 0 \text{ for all }i \in [n], \omega \in \F^d\setminus \block[i]
    \}
  \]
  and the coordinate projection
  \[
    \pi^{\B} : \R^{\B(d,n)} \rightarrow \R^{\B}\,.
  \]
\end{definition}

\begin{remark}
    \label{rem:CTP_face_of_full_CTP}
    For each block configuration $\blockconfig$ in $\F^d$ of length $n$, the map $\pi^{\B}$ induces an isomorphism between the face $\CTP[d,n]\cap \linsubspace[\blockconfig]$ of the full cyclic transversal polytope $\CTP[d,n]$ and $\CTP[\B]$. If $\bar{a}^T\bar{x}\ge b$ is an inequality that is valid for $\CTP[d,n]$, then the \emph{$\blockconfig$-restriction} $a^Tx \ge b$ of $\bar{a}^T\bar{x} \ge b$ with $a = \pi^{\B}(\bar{a})$ is valid for $\CTP[\B]$. If we have
    \[
      \CTP[d,n] = \{\bar{x} \in \TP[d,n] : \bar{A}\bar{x} \ge b\}
    \]
    and $Ax\ge b$ is the system formed by the $\blockconfig$-restrictions of the inequalities in $\bar{A}\bar{x}\ge b$, then
    \[
        \CTP[\B] = \{x \in \TP[\B] : Ax \ge b\}
    \]
    holds.
\end{remark}

\begin{definition}
For a block configuration $\blockconfig$ of length $n$ in $\F^d$ and a linear map $\varphi : \F^d \rightarrow \F^{\bar{d}}$ we consider the block configuration $\varphi(\blockconfig) \coloneqq \blockconfigmod$ of the same length $n$  in $\F^{\bar{d}}$ with $\blockmod[i] \coloneqq \varphi(\block[i])$ for all $i \in [n]$. We refer to the  map
\(
  \Phi: \TP[\B] \rightarrow \R^{\blockmod}
\),
defined via
\[
  \Phi(x)^i_{\bar{\omega}} \coloneqq \sum_{\omega \in \block[i] : \varphi(\omega) = \bar{\omega}} x^i_{\omega}
  \quad\text{for all }i \in [n], \bar{\omega} \in \blockmod[i], x \in \TP[\B]
\]
as the \emph{$\varphi$-induced map} (w.r.t. $\blockconfig$).
\end{definition}

In the above definition, we restrict the domain of $\Phi$ to $\TP[\B]$ in order to have preimages of $\Phi$ being contained in $\TP[\B]$, which is going to simplify notations later.

As $\varphi$ is linear, it maps the cyclic transversals of $\blockconfig$ to some cyclic transversals of $\blockmod$. Hence
we have
\begin{equation}\label{eq:Phi_of_CTP}
    \Phi(\CTP[\blockconfig]) \subseteq \CTP[\blockconfigmod]\,,
\end{equation}
and for any inequality
\begin{equation}\label{eq:lifting_ieq}
  \sum_{i=1}^n \sum_{\bar{\omega} \in \blockmod[i]} \bar{a}^i_{\bar{\omega}}\bar{x}^i_{\bar{\omega}} \ge \beta
\end{equation}
that is valid for $\CTP[\blockconfigmod]$ the inequality
\begin{equation}\label{eq:lifting_lifted_ieq}
  \sum_{i=1}^n \sum_{\bar{\omega} \in \blockmod[i]}\sum_{\omega \in \block[i] : \varphi(\omega) = \bar{\omega}}\bar{a}^i_{\bar{\omega}} x^i_{\omega}\ge \beta
\end{equation}
is valid for $\CTP[\blockconfig]$.
We refer to~\eqref{eq:lifting_lifted_ieq} as the \emph{$\varphi$-lifting of the inequality}~\eqref{eq:lifting_ieq}.

\begin{remark}\label{rem:lifting_ieqs}
    Let $\blockconfig$ be a block configuration in $\F^d$ and $\varphi : \F^d\rightarrow \F^{\bar{d}}$ a linear map with $\varphi$-induced map $\Phi : \TP[\B] \rightarrow \R^{\varphi(\blockconfig)}$. If
    \[
      \CTP[\varphi(\blockconfig)] = \{\bar{x} \in \TP[\varphi(\blockconfig)]:\bar{A}\bar{x}\ge b\}
    \]
    holds and $Ax\ge b$ is the system of inequalities obtained by $\varphi$-lifting the inequalities in $\bar{A}\bar{x}\ge b$, then we have
    \[
      \Phi^{-1}(\CTP[\varphi(\blockconfig)]) = \{x\in\TP[\B]:Ax\ge b\}\,.
    \]
  \end{remark}

  If the linear map $\varphi:\F^d\rightarrow\F^{\bar{d}}$ is injective on $\linspan(\B)$, then $\Phi$ induces an isomorphism between $\CTP[\B]$ and $\CTP[\varphi(\blockconfig)]$, and thus we have $\Phi(\CTP[\blockconfig]) = \CTP[\varphi(\blockconfig)]$ and $\CTP[\B] = \Phi^{-1}(\CTP[\varphi(\blockconfig)])$. In particular we conclude the following.

\begin{remark}\label{rem:block_iso}
    Let $\blockconfig$ be a block configuration in $\F^d$ of length $n$ and rank $r = \rank(\B)$. If
    \[
        \varphi : \F^d \rightarrow \F^{r}
    \]
    is a linear map that induces an isomorphism between $\linspan(\B)$ and $\F^{r}$, then $\Phi^{-1}\circ\pi^{\varphi(\blockconfig)}$ induces an isomorphism between $\CTP[\blockconfig]$ and the face $\CTP[r,n] \cap \linsubspace[\varphi(\blockconfig)]$ of the full cyclic transversal polytope $\CTP[r,n]$.
\end{remark}

\begin{definition}
 For a block configuration $\blockconfig$ in $\F^d$ of length $n$ and a cyclic sequence $\sigma=(\sigma(1),\dots,\sigma(n)) \in (\F^d)^n$ we call the block configuration
\[
  \B+\sigma \coloneqq (\block[1]+\sigma(1),\dots,\block[n]+\sigma(n))
\]
with $\block[i]+\sigma(i) \coloneqq \{\omega + \sigma(i) : \omega \in \block[i]\}$ for each $i \in [n]$ the \emph{$\sigma$-shift} of $\blockconfig$.
\end{definition}

Note that for $\sigma \in \CT[\B]$ each block in $\B+\sigma$ contains $\zerovec$.

Since the sum of two cyclic sequences is a cyclic sequence, we have the following.

\begin{remark}\label{rem:cyclic_shift}
    If $\blockconfig$ is a block configuration and $\sigma$ is a cyclic sequence, then $\CTP[\B]$ and $\CTP[\B+\sigma]$ are affinely isomorphic by means of the coordinate bijection $\Sigma:\R^{\B} \rightarrow \R^{\B+\sigma}$ with
    \[
        \Sigma(x)^i_{\bar{\omega}} \coloneqq x^i_{\bar{\omega} - \sigma(i)} = x^i_{\bar{\omega} + \sigma(i)}
    \]
    for all $i \in [n]$ and $\bar{\omega} \in \block[i]+\sigma(i)$.
\end{remark}

     \section{The expressive power of CTPs}
\label{sec:expressive_power}

\subsection{Satisfiability polytopes and a universality theorem}

\begin{definition}\label{def:kSAT}
    For $k\in\{2,3,\dots\}$ we define a \emph{$k$-SAT formula} to be a Boolean formula $\varphi = \land_{i=1}^p C_i$
    in $q$ variables, for which each clause $C_i$ in the conjunction is a disjunction $\lor_{j=1}^{k_i} \ell_{i,j}$ of $k_i \in [k] $ literals $\ell_{i,j}$ associated with $k_i$ pairwise different variables, where a \emph{literal} is a variable $b$ or the negation $\lnot b$ of a variable.
By identifying the Boolean values $\true$ and $\false$ with the real numbers $1$ and $0$, respectively, such a $k$-SAT formula $\varphi$ induces a function $\varphi : \{0,1\}^q \rightarrow \{0,1\}$, and thus defines a subset
    \[
        T(\varphi) \coloneqq \{x \in \{0,1\}^q : \varphi(1) = 1\}
    \]
    of $0/1$-vectors and the corresponding 0/1-polytope $P(\varphi)\coloneqq\conv(T(\varphi))$.
\end{definition}

We recall that an \emph{extension} of a polytope $P$ is a (usually higher-dimensional) polytope $Q$ such that $P$ can be represented as the image of $Q$ under some affine map. The \emph{size} of such an extension is the number of facets of $Q$. An \emph{extended formulation} of $P$ is a description of an extension of $P$ by means of linear equations and linear inequalities, where the \emph{size} of such an extended formulation is the number of inequalities it comprises. The \emph{extension complexity} of $P$ is the smallest size of any extended formulation of $P$.

\begin{theorem}
    \label{thm:kSAT}
    For each $k$-SAT formula $\varphi$ with $p$ clauses in $q$ non-redundant variables (i.e., every variable appears in the formula), the polytope $P(\varphi)$ has an extension (with the affine map being a coordinate projection) that is a \CTPtext associated with a block configuration of length $q+p$, size at most $2q+p(2^k-1)$ and rank at most $kp$. In case of $k=2$ the coordinate projection induces an isomorphism between the \CTPtext and $P(\varphi)$.
\end{theorem}

\begin{proof}
  For the sake of notational simplicity we only formulate the proof for the case of every clause having exactly $k$ literals; the more general case can be proven by an easy adaption.

    For $d \coloneqq pk$ we identify $\F^d$ with $\F^{p \times k}$, where the $i$-th row of a matrix in $\F^{p\times k}$ is associated with the clause $C_i = \lor_{j=1}^k \ell_{i,j}$ of $\varphi$, and the $j$-th entry in that row corresponds to the literal $l_{i,j}$ in $C_i$.

    Denoting the Boolean variables in $\varphi$ by $b_1,\dots,b_q$ we define, for each $v \in [q]$, two vectors $\omega(v,\true), \omega(v,\false) \in \F^d$ (representing the possibilities to substitute $b_v$ for $\true$ and $\false$, respectively) that have all zero-components with the following exceptions:
    \begin{eqnarray*}
        \omega(v,\true)_{i,j} &\coloneqq 1 & \text{if } \ell_{i,j} = b_v \\
        \omega(v,\false)_{i,j} &\coloneqq 1 & \text{if } \ell_{i,j} = \lnot b_v
    \end{eqnarray*}
    We choose the first $q$ blocks to be
    \[
      \B(v) \coloneqq \{\omega(v,\true), \omega(v,\false)\}
      \quad (v \in [q]\,)\,.
    \]

    For $i \in [p]$ and $\tau \in \F^k$ denote by $\omega(i,\tau) \in \F^{p \times k}$ the matrix that has $\tau$ as its $i$-th row and zeros everywhere else (representing a $\true$-$\false$ pattern of $(\ell_{i,1},\dots,\ell_{i,k})$). The remaining blocks are chosen to be
    \[
      \B(q+i) \coloneqq \{\omega(i,z) : z \in \F^k\setminus\{\zerovec\}\}
        \quad (i \in [p])\,.
    \]

    It is easy to see (note that since no variable is redundant we have $\omega(v,\true) \ne\omega(v,\false)$ for all $v \in [q]$) that the projection to the coordinates associated with the variables $\omega(v,\true)$ ($v \in [q]$) induces a one-to-one map between the vertices of $\CTP[\blockconfig]$ (with $\blockconfig = (\B(1),\dots,\B(q+p))$) and $P(\varphi)$, hence the former polytope is an extension of the latter one.

    In order to see that in case of $k=2$ that projection is an isomorphism, we first observe that the transversal equations
    \[
        x^v_{\omega(v,\false)} = 1 - x^v_{\omega(v,\true)}
        \quad (v \in [q])
    \]
    hold for $\CTP[\blockconfig]$ (recall $\omega(v,\true) \ne\omega(v,\false)$). While those equations hold regardless of $k$, for $k=2$ we find that furthermore for each $i \in [p]$, the non-singular system
    \[
        \begin{array}{rcrcrcl}
            x^{q+i}_{\omega(i,(1,0))} &
            &
            &
            + &
            x^{q+i}_{\omega(i,(1,1))}
            & =
            h_1 \\
            &
            &
            x^{q+i}_{\omega(i,(0,1))} &
            + &
            x^{q+i}_{\omega(i,(1,1))}
            & =
            h_2 \\
            x^{q+i}_{\omega(i,(1,0))} &
            + &
            x^{q+i}_{\omega(i,(0,1))} &
            + &
            x^{q+i}_{\omega(i,(1,1))}
            & =
            1
        \end{array}
    \]
    with
    \[
        h_1 \coloneqq
        \begin{cases}
            x^v_{\omega(v,\true)} & \text{if } \ell_{i,1} = b_v \\
            x^v_{\omega(v,\false)} & \text{if } \ell_{i,1} = \lnot b_v \\
        \end{cases}
        \quad\text{and}\quad
        h_2 \coloneqq
        \begin{cases}
            x^v_{\omega(v,\true)} & \text{if } \ell_{i,2} = b_v \\
            x^v_{\omega(v,\false)} & \text{if } \ell_{i,2} = \lnot b_v \\
        \end{cases}
    \]
    is valid for $\CTP[\blockconfig]$, implying that for the points in $\CTP[\blockconfig]$ all coordinates depend affinely on those ones we project on. Hence the projection induces an affine isomorphism between $\CTP[\blockconfig]$ and its image $P(\varphi)$.
\end{proof}

Obviously, the construction of a block configuration described in the proof of Theorem~\ref{thm:kSAT} can be carried out in polynomial time, which implies the following.

\begin{corollary}
    Deciding whether a given block configuration has a cyclic trans\-versal (or, equivalently, whether its \CTPtext is non-empty) is \(\NP\)-complete.
\end{corollary}

Furthermore, as it is well-known that even for 2-SAT formulas $\varphi$ the smallest size of an extension of $P(\varphi)$ cannot be bounded polynomially in the size of $\varphi$ (see~\cite{AvisTiwary2015}, who derive this from the seminal work~\cite{Fiorini_etal15}) we also conclude the following from Theorem~\ref{thm:kSAT}.

\begin{corollary}
    Cyclic transversal polytopes in general do not admit extended formulations whose sizes can be bounded polynomially in the sizes of the block configurations.
\end{corollary}

Since every polytope $P$ with $k$ vertices is a linear image of a $(k-1)$-dimensional simplex, Example~\ref{ex:simplex} implies that some \CTPtext associated with a block configuration of size $\bigO(k)$ provides an extension for $P$. Of course, usually $k$ is not bounded polynomially in the encoding length of the problem one aims to solve. However, it will turn out that for most 0/1-polytopes that are relevant in combinatorial optimization, Theorem~\ref{thm:kSAT} opens up possibilities to represent them by extensions that are \CTPtext{s} associated with block configurations whose sizes are bounded polynomially in the encoding length of the problem the 0/1-polytope is associated with.

We denote by $\{0,1\}^{\star} = \cup_{\ell=1}^{\infty} \{0,1\}^{\ell}$ the set of all 0/1-vectors of finite length, and by $\len(x)=\ell$ for $x \in \{0,1\}^{\ell}$ the length of $x$.

\begin{definition}\label{def:wellencodedfamily}
    Let $Y \subseteq \{0,1\}^{\star}$ be some set of 0/1-vectors (playing the role of the encodings of the instances of a certain problem) as well as $m(y) \in \{1,2,\dots\}$ and $X(y) \subseteq \{0,1\}^{m(y)}$ for each $y \in Y$. We call $\{(y,\conv(X(y))) : y \in Y\}$ a family of \emph{well-encoded 0/1-polytopes} if there is a polynomial $f$ such that $m(y) \le f(\len(y))$ holds for each $y \in Y$ and the problem to decide for a given pair $(y,x)$ of 0/1-vectors whether $y \in Y$ and $x \in X(y)$ holds is in $\NP$.
\end{definition}

As an example we may construct $Y$ such that it encodes all undirected graphs (e.g., by means of binary encodings of adjacency matrices) and $X(y)$ to consist of all incidence vectors of Hamiltonian cycles in the graph encoded by $y$, in which case the relevant membership problem is even solvable in polynomial time (rather than being in $\NP$ only). In this example, the polytopes in the well-encoded family would be the traveling salesman polytopes.
In another example, $X(y)$ may consist of all incidence vectors of Hamiltonian subgraphs of the graph encoded by $y$ (thus we deal with the Hamiltonian subgraph polytopes), in which case that membership problem still is in $\NP$ (but presumably not polynomial-time solvable).

\begin{theorem}\label{thm:universality}
    For each well-encoded family $\{(y,P(y)) : y \in Y\}$ of 0/1-polytopes there is a polynomial $g$ such that for every $y \in Y$ there is a block configuration $\B_y$ of size at most $g(\len(y))$ such that $\CTP[\B_y]$ is an extension of $P(y)$ (by means of a coordinate projection).
\end{theorem}

\begin{proof}
    We continue to use the notations from Definition~\ref{def:wellencodedfamily} (in particular, $P(y) = \conv(X(y))$) and identify the Boolean values $\true$ and $\false$ with the real numbers $1$ and $0$, respectively. Then a proof of the Cook-Levin Theorem as it is presented, e.g., in \cite[Lem.~2.11, 2.4]{AroraBarak09} shows that there is a polynomial $h$ such that for every 0/1-vector $y \in Y$ there is a 3-SAT formula $\varphi$ with at most $h(\len(y)+m(y))$ clauses (and at most that many variables) such that for every $x \in \{0,1\}^{m(y)}$ we have  $x \in X(y)$ if and only if there is some $\true$/$\false$ vector $z$ (of length at most $h(\len(y)+m(y)))$ with $\varphi(y,x,z) = \true$, hence $P(y)$ is isomorphic to the projection of $P(\varphi)$ to the $x$-coordinates. Due to Theorem~\ref{thm:kSAT} (and the fact that $m(y)$ is bounded polynomially in $\len(y)$) this proves the claim.
\end{proof}

While Theorem~\ref{thm:universality} asserts that polytopes arising in combinatorial optimization can be represented as projections of \CTPtext{s} in a non-trivial way, we next are going to investigate some further cases of well-known polytopes that turn out to even be affinely isomorphic to \CTPtext{s}, thus allowing for immediate transfer of knowledge on the latter ones.

\subsection{Packing polytopes}

Recall that a stable set in a graph $G=(V,E)$ is a subset of pairwise non-adjacent nodes and the \emph{stable-set polytope} $\STAB[G] \subseteq \R^V$ is the convex hull of the incidence vectors of stable sets in $G$. Introducing a Boolean variable for each node it is easy to see (and well-known) that we have $\STAB[G] = P(\varphi)$ for the $2$-SAT formula that for each edge $e \in E$ has a clause formed by the two negative literals associated with the end-nodes of $e$. Hence, Theorem~\ref{thm:kSAT} implies the following (where the reason for excluding isolated nodes results from the restriction to non-redundant variables in that theorem).

\begin{remark}\label{rem:STAB}
  For each graph $G=(V,E)$ without isolated nodes
    there is a block configuration of length $|V|+|E|$ and size $2|V|+3|E|$ whose associated \CTPtext is affinely isomorphic to $\STAB[G]$.
\end{remark}

Let $\mathcal{H}$ be a finite family of sets (where a set may occur several times in the family). A \emph{packing} in $\mathcal{H}$ is a sub-family of $\mathcal{H}$ in which each pair of sets has empty intersection. The convex hull $\PACK[\mathcal{H}]$ of the incidence vectors (from $\{0,1\}^{\mathcal{H}}$) of packings in $\mathcal{H}$ is called a \emph{set-packing polytope}. The stable-set polytope $\STAB[G]$ for a graph $G$ with every connected component having at least three nodes is affinely isomorphic to the set-packing polytope associated with the family formed by the stars of the nodes of $G$ (where the \emph{star} of a node is the set of all edges incident to that node; note that the imposed condition on the connected components implies that no two nodes have the same star). Conversely, for any family $\mathcal{H}$ of sets we can construct its \emph{conflict graph} $G(\mathcal{H})$, where the sets in $\mathcal{H}$ are the nodes, two of which being adjacent if and only if they intersect. We then have $\PACK[\mathcal{H}]=\STAB[G(\mathcal{H})]$. Therefore, Remark~\ref{rem:STAB} shows that $\PACK[\mathcal{H}]$ is affinely isomorphic to a \CTPtext associated with a block configuration of size $\bigO(|\mathcal{H}|^2)$. The next theorem describes a construction that in general reduces that size significantly.

\begin{theorem}
    \label{thm:packing}
    For each finite family $\mathcal{H}$ of non-empty sets and $E=\cup_{H \in \mathcal{H}}H$ there is a block configuration $\B$ of length $|\mathcal{H}|+|E|$ and size $2|\mathcal{H}|+|E| + \sum_{H \in \mathcal{H}}|H|$ such that
    $\CTP[\blockconfig]$ and $\PACK[\mathcal{H}]$ are affinely isomorphic.
\end{theorem}

\begin{proof}
    With $d = \sum_{h \in \mathcal{H}} |H|$ we consider the vectors $\omega \in \F^d$ as tables whose rows correspond to the sets in $\mathcal{H}$, where the row corresponding to $H \in \mathcal{H}$ has $|H|$ entries, addressed as $\omega_{H,1},\dots,\omega_{H,|H|}$ (recall that we have $|H| \ge 1$).
    For every $H \in \mathcal{H}$ we fix an arbitrary order $H=\{e_1,\dots,e_{|H|}\}$ of the elements in $H$, denote by $\omega(H)\in \F^d$ the vector that has all components equal to zero except for $\omega(H)_{H,1}=1$, and define for each $j \in [|H|]$ the vector $\omega(H,e_{j}) \in \F^d$ to be the table that has all entries equal to zero with the following exceptions:
\begin{eqnarray*}
    \omega(H,e_{j})_{H,j} \coloneqq \omega(H,e_{j})_{H,j+1} & \coloneqq 1 & \text{for all } j = 1, \dots, |H| - 1 \\
    \omega(H,e_{|H|})_{H,|H|} & \coloneqq 1
\end{eqnarray*}
We make the following observation which we will refer to several times in the remaining part of the proof.
\begin{nestedenvironment}{theorem}
\begin{remark}\label{rem:packCTP}
    The sum of a subset $\Omega$ of the vectors $\omega(H)$ and $\omega(H,e)$ ($H \in \mathcal{H}, e \in H$) equals $\zerovec\in\F^d$ if and only if for each $H \in \mathcal{H}$ and $e \in H$ we have
    \[
      \omega(H,e) \in \Omega \quad \Leftrightarrow \quad \omega(H) \in \Omega\,.
    \]
\end{remark}
\end{nestedenvironment}

We now define the block configuration $\blockconfig$ consisting of $|\mathcal{H}|+|E|$ blocks by including for each $H \in \mathcal{H}$ the block
\[
  \B(H) \coloneqq \{\zerovec,\omega(H)\}
\]
as well as
for each $e \in E$ the block
\[
    \B(e) \coloneqq \{\zerovec\} \cup \{\omega(H,e) : H \in \mathcal{H}, e \in H\}\,.
\]

If $\xi\in\CT[\B]$ is any cyclic transversal of $\B$, then the sub-family $\mathcal{H}' = \{H \in \mathcal{H} : \xi(H) = \omega(H)\}$ is a packing, since Remark~\ref{rem:packCTP} implies for every pair $H_1,H_2 \in \mathcal{H}'$ with $e \in H_1 \cap H_2$ that we have both $\xi(e) = \omega(H_1,e)$ and $\xi(e) = \omega(H_2,e)$, thus $H_1=H_2$. Conversely, if $\mathcal{H'} \subseteq \mathcal{H}$ is a packing then setting
\[
    \xi(H) \coloneqq
    \begin{cases}
        \omega(H) & \text{if } H \in \mathcal{H}' \\
        \zerovec & \text{otherwise}
    \end{cases}
\]
for all $H \in \mathcal{H}$ and
\[
  \xi(e) \coloneqq
  \begin{cases}
    \omega(H,e) & \text{if }e \in H \text{ for some (then unique) }H\in \mathcal{H}'\\
    \zerovec & \text{otherwise}
  \end{cases}
\]
for all $e \in E$ defines a cyclic transversal of $\B$.

Consequently, the projection to the coordinate subspace associated with the variables $x^{H}_{\omega(H)}$ ($H \in \mathcal{H}$) maps $\CTP[\blockconfig]$ to a polytope that is affinely isomorphic to $\PACK[\mathcal{H}]$. Since due to Remark~\ref{rem:packCTP} the equations
\[
  x^e_{\omega(H,e)} = x^H_{\omega(H)}
  \quad (H \in \mathcal{H}, e \in H)
\]
as well as the transversal equations
\[
  x^H_{\zerovec} + x^H_{\omega(H)} = 1 \quad(H \in \mathcal{H})
\]
are valid for $\CTP[\blockconfig]$, that projection in fact induces an affine isomorphism between $\CTP[B]$ and $\PACK[\mathcal{H}]$.
\end{proof}

If each member in the set family $\mathcal{H}$ has at least two elements, then one can decrease the size of the block configuration in the proof of Theorem~\ref{thm:packing} by defining $d \coloneqq \sum_{h \in \mathcal{H}} (|H|-1)$,
changing the definition of the non-zero components of the $\omega(\cdot,\cdot)$-vectors to
\begin{eqnarray*}
    \omega(H,e_1)_{H,1} & \coloneqq 1 \\
    \omega(H,e_{j})_{H,j-1} \coloneqq \omega(H,e_{j})_{H,j} & \coloneqq 1 &\text{for all }j = 2,\dots,|H|-1 \\
    \omega(H,e_{|H|})_{H,|H|-1} & \coloneqq 1
\end{eqnarray*}
for each $H \in \mathcal{H}$, using only the blocks $\B(e)$ as defined in the proof for $e \in E$, and
choosing an arbitrary representative $e(H) \in H$ for each $H \in \mathcal{H}$. Then
the projection to the coordinate-subspace associated with the variables $x^{e(H)}_{\omega(H,e(H))}$ (for all $H \in \mathcal{H}$) yields an affine isomorphism between $\CTP[\blockconfig]$ and $\PACK[\mathcal{H}]$. This establishes the following.

\begin{theorem}
    \label{thm:packing_smaller}
    For each finite family $\mathcal{H}$ of sets with $|H| \ge 2$ for every $H \in \mathcal{H}$, there is a block configuration $\B$ of size $|E| + \sum_{H \in \mathcal{H}}|H|$ with $E=\cup_{H \in \mathcal{H}}H$ such that
    $\CTP[\blockconfig]$ and $\PACK[\mathcal{H}]$ are affinely isomorphic.
\end{theorem}

As a matching in a graph is a packing of edges, Theorem~\ref{thm:packing_smaller} implies the following, where the set $E$ in the formulation of the theorem is the node set of the graph and the set family $\mathcal{H}$ is its edge set. Note that the part of the statement on perfect matchings refers to the modification that arises from removing $\zerovec$ from each block.

\begin{corollary}\label{cor:matching}
    For every graph $G=(V,E)$ there are block configurations $\B^{\match(G)}$ and $\B^{\perfmatch(G)}$ of size $|V|+2|E|$ and $2|E|$, respectively, such that $\CTP[\B^{\match(G)}]$ and $\CTP[\B^{\perfmatch(G)}]$ are affinely isomorphic to
    the matching polytope and to the perfect matching polytope, respectively, of $G$.
\end{corollary}

\subsection{Limits of CTPs}

As we have seen, many important 0/1-polytopes like cut polytopes and stable set polytopes are affinely isomorphic to \CTPtext{s}. Nevertheless, \CTPtext{s} have a rather special combinatorial structure. For instance, the following gives a necessary condition for \CTPtext{s}.

\begin{theorem}\label{thm:ctp_necessary_cond}
    Every \CTPtext whose number of vertices is not a power of two has the property that every pair of vertices is contained in a proper face of the polytope.
\end{theorem}

\begin{proof}
  Let \(P\) be any \CTPtext that has a pair of vertices that does \emph{not} lie in a common proper face of \(P\). It suffices to show that \(P\) is isomorphic to the cycle polytope of a binary matroid (see Example~\ref{ex:binary_matroids}), which in turn implies that the number of vertices of \(P\) is a power of two.

  Let \(\B = (\B(1),\dots,\B(n))\) be a block configuration over \(\venuespace\) of minimal size such that
  \(\CTP[\blockconfig]\) is isomorphic to $P$. Note that we have $n\ge 3$ since $P$ is not a simplex (see Example~\ref{ex:simplex}). In fact, we can assume $P=\CTP[\blockconfig]$.
  Let \(v\) and \(\overlinemod{v}\) be two vertices of \(\CTP[\blockconfig]\) that do not lie in a common proper face of \(\CTP[\blockconfig]\)
with corresponding cyclic transversals \(\blockelement\) and \(\blockelementmod\).
  By shifting (see Remark~\ref{rem:cyclic_shift}) we can assume that \(\blockelementmod\) is the zero-transversal, i.e., we have \(\blockelementmod[i] = \zerovec\) for all $i \in [n]$. We observe that $\xi(i) = \zerovec$ does not hold for any $i \in [n]$, since that would imply that both $v$ and \(\overlinemod{v}\) satisfy the equation $x^i_{\zerovec} = 1$, which defines a proper face of $\CTP[\blockconfig]$ due to the minimality assumption on $\B$ (if $x^i_{\zerovec} = 1$ holds for the entire polytope $\CTP[\blockconfig]$ then removing the $i$-th block results in a \CTPtext that is isomorphic to $\CTP[\blockconfig]$, recall that we have $n \ge 3$).

Thus we have $\{\blockelement[i],\zerovec\} \subseteq \B(i)$ with $\xi(i) \ne \zerovec$ for each $i \in [n]$. In fact, $\{\blockelement[i],\zerovec\} = \B(i)$ holds for all $i \in [n]$, since for every $\omega \in \B(i) \setminus \{\blockelement[i],\zerovec\}$ the equation $x^i_{\omega}=0$ would be satisfied by both \(v\) and \(\overlinemod{v}\), which again due to the minimality property of $\B$ does not hold for the entire polytope $\CTP[\blockconfig]$.

With the matrix $M \in \F^{d\times n}$ whose columns are $\blockelement[1],\dots,\blockelement[n] \in \F^d\setminus\{\zerovec\}$ we then find that $\CTP[\blockconfig]$ is isomorphic to $\BSP[M]$ (see Example~\ref{ex:binary_matroids}).
\end{proof}

\begin{remark}\label{rem:ctp_necessary_cond}
From the proof of Theorem~\ref{thm:ctp_necessary_cond} we see that every \CTPtext which has a pair of vertices that is not contained in a common proper face, is in fact affinely isomorphic to a binary subspace polytope.
\end{remark}

Using Theorem~\ref{thm:ctp_necessary_cond} we can now give some examples of polytopes that are not isomorphic to \CTPtext{s}.

\begin{corollary}\label{cor:uniform_matroids}
  For each $r \ge 1$, the basis polytope (\ie, the convex hull of the incidence vectors of the bases) of the uniform matroid $U_{r,2r}$ of rank $r$ on $2r$ elements is not isomorphic to a \CTPtext.
\end{corollary}

\begin{proof}
  As the bases of $U_{r,2r}$ are the $r$-element subsets of $[2r]$, the vertices of the basis polytope $P$ of $U_{r,2r}$ come in pairs for all of which the middle point of the line segment joining the two vertices forming the pair is $\frac12\onevec$, the centroid of the vertices of $P$, thus a point in the relative interior of $P$. Consequently, none of those pairs is contained in a proper face of $P$.

  Moreover, the number of vertices of $P$ equals
  \((2r)\cdot(2r-1)\cdots(r+1)/(r!)\), which is not a power of two, since at least one of the factors in the numerator is a prime according to Bertrand's postulate (see, e.g., \cite[Chap. 2]{AignerZiegler2018}).

  Thus the claim follows from Theorem~\ref{thm:ctp_necessary_cond}.
\end{proof}

Similar reasoning as used in the proof of Corollary~\ref{cor:uniform_matroids} yields the following examples of polytopes that are not isomorphic to \CTPtext{s}.

\begin{example}
  Let \(q \geq 3\) be a number that is not a power of two. Then the \(q\)-dimensional cross polytope \(\CROSS[q] \coloneqq \conv\{\pm\unitvec_1,\dots,\pm\unitvec_q\} \subseteq \R^q\) is not isomorphic to a \CTPtext.

  Maybe somewhat surprisingly, whenever \(q\) is a power of two, \(\CROSS[q]\) indeed is affinely isomorphic to a \CTPtext. A constructive proof of that statement can be found in \cite[Proof of Thm.~3.24]{Frede2023}.

\end{example}

We eventually touch briefly the question whether traveling salesman polytopes are isomorphic to \CTPtext{s}.

\begin{example}\label{ex:tsp}
  Let us denote by \(\TSP[q]\) the \emph{traveling salesman polytope} associated with the complete (undirected) graph $K_q$ on the node set $[q]$. Clearly, $\TSP[3]$ is a single point and $\TSP[4]$ has three vertices. They both are simplices, hence they are affinely isomorphic to \CTPtext{s} (see Example~\ref{ex:simplex}). The polytope $\TSP[5]$ has $12$ vertices and is defined by the degree equations and the zero-one bounds on the variables~\cite{Norman1955,GroetschelPadberg1979a}. Since the Hamiltonian cycles corresponding to the node-sequences $(1,2,3,4,5)$ and $(1,3,5,2,4)$ partition the edge set of $K_5$ into two disjoint subsets, Theorem~\ref{thm:ctp_necessary_cond} implies that $\TSP[5]$ is not isomorphic to any \CTPtext. We conjecture that a similar statement holds for $\TSP[q]$ with $q\ge 6$.
\end{example}

     \section{Equations and inequalities for CTPs}
\label{sec:eq_and_ieneq}

Since  the \CTPtext{s} for block configurations of length two are simplices (see Example~\ref{ex:simplex}), we restrict our attention to block configurations $\B$ in $\F^d$ of length $n \ge 3$ in the following.

As we observed in Section~\ref{sec:concept}, the full \CTPtext $\CTP[1,n]$ is isomorphic to the parity polytope
\(
    \EVEN[n]
\).
We thus can translate some well-known facts about $\EVEN[n]$
(first observed by Jeroslow~\cite{Jeroslow75}) into facts about $\CTP[1,n]$.

\begin{remark}\label{rem:CTP1n}
  For all $n \ge 3$, the following statements hold:
  \begin{enumerate}
    \item $\CTP[1,n]$ is the solution set to the transversal equations $x^i_0 + x^i_1 = 1$ for all $i \in [n]$, the nonnegativity constraints $x^i_0,x^i_1 \ge 0$ for all $i \in [n]$, and the
    \emph{odd-set inequalities}
    \begin{equation}\label{eq:odd_set_ieq}\tag{OS}
      \sum_{i \in I} x^i_0 +\sum_{i \in [n]\setminus I} x^i_1 \ge 1
    \end{equation}
    for all subsets $I \subseteq [n]$ of odd cardinality $|I| \in 2\Z+1$.
    \item $\CTP[1,n]$ has dimension $n$.
    \item The nonnegativity constraints define facets of $\CTP[1,n]$ for $n\ge 4$.
    \item The odd-set inequalities define facets of $\CTP[1,n]$.
  \end{enumerate}
\end{remark}

Let $\blockconfig$ be any block configuration in $\F^d$ of length $n$, $\eta \in \F^d\setminus\{\zerovec\}$ some non-zero vector, $\varphi : \F^d \rightarrow\F$ the linear form with $\varphi(\omega) = \eta^T\omega$ for all $\omega\in\F^d$, $\blockconfigmod = \varphi(\blockconfig)$, and $I \subseteq [n]$ a subset of odd cardinality. Then the $\blockconfigmod$-restriction of~\eqref{eq:odd_set_ieq} is valid for $\CTP[\blockconfigmod]$ (see Remark~\ref{rem:CTP_face_of_full_CTP}). Hence the $\varphi$-lifting
\begin{equation}\label{eq:lifted_odd_set_ieq}\tag{LOS}
  \sum_{i \in I} \sum_{\omega\in \block[i]: \eta^T\omega = 0}x^i_{\omega} +\sum_{i \in [n]\setminus I} \sum_{\omega\in \block[i]: \eta^T\omega = 1}x^i_{\omega} \ge 1
\end{equation}
of the $\blockconfigmod$-restriction of~\eqref{eq:odd_set_ieq} is valid for $\CTP[\blockconfig]$. We call~\eqref{eq:lifted_odd_set_ieq} the \emph{lifted odd-set inequality} (\emph{LOS inequality}) for $\CTP[\blockconfig]$ associated with $\eta$ and $I$. Clearly, the LOS inequality for $\CTP[\blockconfig]$ associated with $\eta$ and $I$ is the $\blockconfig$-restriction of the LOS inequality for $\CTP[d,n]$ associated with $\eta$ and $I$. The LOS inequalities for $\CTP[1,n]$ are the odd-set inequalities~\eqref{eq:odd_set_ieq}.

For every $j \in [d]$
and $(\xi(1),\dots,\xi(n)) \in \block[1]\times\cdots\times\block[n]$ whose incidence vector satisfies~\eqref{eq:lifted_odd_set_ieq} for $\eta=\unitvec_j \in\F^d$ (the $j$-th standard basis vector having a $0$ in each component except for a $1$ in the $j$-th one) we have $\sum_{i=1}^n \xi(i)_j = 0$ for all $j \in [d]$, which establishes the following observation.

\begin{proposition}
    For every block configuration $\blockconfig$ in $\F^d$, the polytope $\CTP[\blockconfig]$ is the convex hull of all integral points in the transversal polytope $\TP[\B]$ (i.e., all 0/1-vectors fulfilling the transversal equations) that satisfy all LOS inequalities associated with the vectors $\unitvec_1,\dots, \unitvec_d \in \F^d$.
\end{proposition}

Before we investigate properties of LOS inequalities in more detail, we make a few observations regarding the high degree of symmetry of full \CTPtext{s} $\PIP[d,n]$.
For two cyclic transversals $\sigma, \xi \in \CT[d,n]$ their sum $\xi + \sigma = (\xi(1)+\sigma(1),\dots,\xi(n)+\sigma(n)) \in \CT[\B(d,n)]$ is a cyclic transversal as well, which we call
the \emph{$\sigma$-shift} of $\xi$ (or, by symmetry, the \emph{$\xi$-shift} of $\sigma$). Clearly, defining addition that way we impose a group structure on $\CT[d,n]$ with the
zero-transversal serving as the neutral element and each cyclic transversal being its own inverse. For each $\sigma \in \CT[d,n]$, we refer to the map $\Sigma$ defined in Remark~\ref{rem:cyclic_shift} (with $\B = \B+\sigma = \B(d,n)$)---that here is
the coordinate permutation $\Sigma : \R^{\B(d,n)} \rightarrow \R^{\B(d,n)}$ with
\[
  \Sigma (x)^i_{\omega} = x^{i}_{\omega-\sigma(i)} = x^{i}_{\omega+\sigma(i)}
\]
for all $i \in [n]$ and $\omega\in\F^d$---as the \emph{$\sigma$-shift}. Note that we have $\Sigma^{-1} = \Sigma$. Clearly, $\Sigma$ maps the incidence vector of a cyclic transversal to the incidence vector of the $\sigma$-shift of that cyclic transversal. Consequently, we have
\[
  \Sigma (\PIP[d,n]) = \PIP[d,n]
\]
for each $\sigma \in \CT[d,n]$ (which is a special case of Remark~\ref{rem:cyclic_shift}). For every pair $\xi^1,\xi^2 \in \CT[d,n]$ the $\sigma$-shift with $\sigma = \xi^1+\xi^2$ maps the vertex of $\CTP[d,n]$ corresponding to $\xi^1$ to the vertex corresponding to $\xi^2$, and vice versa. Consequently, $\PIP[d,n]$ looks the same at each vertex in the sense that the radial cones at all vertices are pairwise isometric (via coordinate permutations). Phrased differently, the additive group $\CT[d,n]$ acts vertex-transitively on $\CTP[d,n]$ via the shifting operation.

\begin{remark}\label{rem:shift}
  For each inequality $a^Tx \ge \beta$ that defines a face $F$ of $\PIP[d,n]$ and for each $\sigma \in \CT[d,n]$ with $\sigma$-shift $\Sigma$ the \emph{$\sigma$-shift} $\Sigma(a)^Tx \ge \beta$ of $a^Tx\ge \beta$ defines the face $\Sigma (F)$ of $\PIP[d,n]$ (recall $\Sigma ^{-1} = \Sigma$). In particular, the $\sigma$-shift of a facet-defining inequality of $\CTP[d,n]$ is a facet-defining inequality as well.
\end{remark}

As an example for shifting an inequality, let $a^Tx \ge 1$ be the LOS inequality for $\PIP[d,n]$ associated with $\eta \in \F^d \setminus\{\zerovec\}$ and the odd subset $I \subseteq [n]$. For any $\sigma \in \CT[d,n]$ the $\sigma$-shift of $a^T x \ge 1$ is the LOS inequality associated with the same $\eta$ and the odd subset
\[
  I \, \triangle \, \{i \in [n]: \eta^T \sigma(i) = 1\}\,,
\]
where $A \triangle B$ denotes the symmetric difference $(A \setminus B) \cup (B \setminus A)$ of $A$ and $B$.

\begin{theorem}
    \label{thm:dim_and_facets}
    The following statements hold for all full \CTPtext{s} with $d \ge 1$ and $n \ge 3$.
    \begin{enumerate}
        \item The affine hull of $\PIP[d,n]$ is described by the transversal equations~\eqref{eq:transversal}.
        \item We have $\dim(\PIP[d,n]) = n(2^d-1)$.
        \item The nonnegativity inequalities define facets of $\PIP[d,n]$, except for the case $(d,n) = (1,3)$.
        \item The LOS inequalities define facets of $\PIP[d,n]$.
    \end{enumerate}
\end{theorem}

\begin{proof}Throughout the proof, for $a \in \R^{\B(d,n)}$ and $\beta \in \R$ we will denote by
  \[
    R(a,\beta) \coloneqq \{\xi \in \CT[d,n] : \sum_{i=1}^n a^i_{\xi(i)} = \beta \}
  \]
  the set of all cyclic transversals of $\B(d,n)$ whose incidence vectors satisfy the equation $a^Tx=\beta$.

  \paragraph{Proof of Statement~1.}
    Let us assume that $a^Tx = \beta$ is an equation
    that holds for $\PIP[d,n]$, i.e., we have $\CT[d,n]\subseteq R(a,\beta)$. We need to show that $a$ is a linear combination of the coefficient vectors of the transversal equations. By subtracting appropriate multiples of those coefficient vectors we may assume $a^i_{\zerovec{}} = 0$ for all $i \in [n]$. Plugging the incidence vector of the \emph{zero-transversal} $\zerovec \in \CT[d,n]$ with $\zerovec(i)=\zerovec$ for all $i \in [n]$ into $a^Tx = \beta$ we find $\beta=0$.
    It suffices to show $a=\zerovec$. Towards this end, we observe that for all $p,q \in [n]$ with $p \ne q$ and for all $\omega\in\F^d$, substituting for $x$ in $a^Tx=0$ the incidence vector of the cyclic transversal $\xi \in \CT[d,n]$ with $\xi(p)=\xi(q)=\omega$ and $\xi(i) = \zerovec$ for all other $i \in [n]\setminus\{p,q\}$ shows $a^p_{\omega}=-a^q_{\omega}$. Hence, choosing for each $i \in [n]$ some $j,k \in [n] \setminus \{i\}$ with $j \ne k$ (such a pair exists due to $n \ge 3$) we find
    \[
      a^i_{\omega}=-a^j_{\omega}=a^k_{\omega}=-a^i_{\omega}\,,
    \]
    thus $a^i_{\omega}=0$ for all $\omega\in \F^d$, which completes the proof of the first statement.

    \paragraph{Proof of Statement~2.}
    This is an immediate consequence of the first statement.

    \paragraph{Proof of Statement~3.}
  By renumbering the blocks it suffices to establish (for $(d,n) \ne (1,3)$) that $x^{1}_{\omega} \ge 0$ defines a facet of $\PIP[d,n]$ for all $\omega \in \F^d$. In fact, applying the $\sigma$-shift with $\sigma(1)=\sigma(2)=\onevec{}+\omega$ and $\sigma(i) = \zerovec$ for all $i \in [n]\setminus\{1,2\}$ shows that we only have to establish the claim for $x^1_{\onevec{}} \ge 0$.

  Towards this end, let $R \coloneqq \{\xi \in \CT[d,n] : \xi(1) \ne \onevec{}\}$ be the set of all cyclic transversals whose incidence vectors are vertices of the face defined by $x^1_{\onevec{}}\ge 0$. As that face clearly is not equal to $\PIP[d,n]$ (e.g., the cyclic transversal $\xi\in \CT[d,n]$ with $\xi(1)=\xi(2)=\onevec$ and $\xi(i) = \zerovec$ for all $i \in [n]\setminus\{1,2\}$ is not contained in $R$), it suffices to show that
    each equation $a^Tx = \beta$ with $R \subseteq R(a,\beta)$ is a linear combination of the transversal equations and the equation $x^1_{\onevec{}}=0$.

    Adding an appropriate linear combination of those $n+1$ equations to $a^Tx = \beta$ we obtain an equation (also denoted by $a^Tx = \beta$, still satisfying $R \subseteq R(a,\beta)$) with $a^i_{\zerovec{}}=0$ for all $i \in [n]$ and $a^1_{\onevec{}}=1$, for which it now suffices to show that all other coefficients are zeros. Again, plugging the zero-transversal $\zerovec \in R$ into the equation reveals $\beta = 0$. For each $i \in [n]$ and $\omega \in \F^d\setminus \{\onevec\}$ we establish $a^i_{\omega}=0$ by arguments similar to those used in the proof of the first statement (as every cyclic transversal $\xi\in \CT[d,n]$ with $\xi(i) \ne \onevec$ for all $i \in [n]$ is contained in $R$).

    It remains to show $a^i_{\onevec{}}=0$ for all $i \in [n] \setminus \{1\}$. In case of $n \ge 4$, we can again proceed as we did in the proof of the first statement by choosing for each $i \in [n]\setminus\{1\}$ both $j \ne k$ from $[n]\setminus\{1,i\}$.
    For $n=3$  we have $d \ge 2$ (in the statement to be proven the case $(d,n)=(1,3)$ was excluded). For $i \in \{2,3\}$ we then denote by $j$ the index with $\{i,j\}=\{2,3\}$ and consider the cyclic transversal $\xi \in \CT[d,n]$ with
    \[
      \xi(i) = \onevec{}, \xi(1) = \unitvec_1, \text{ and } \xi(j) = \onevec{}+\unitvec_1\,.
    \]
    We have $\unitvec_1 \ne \onevec{}$ due to $d \ge 3$. Thus (recall $i \ne 1$)
    $\xi \in R$ and hence $a^i_{\onevec{}} + a^1_{\unitvec_1} + a^j_{\onevec{}+\unitvec_1} = 0$ holds with $a^1_{\unitvec_1}=a^j_{\onevec{}+\unitvec_1} =0$, which implies $a^i_{\onevec{}}=0$.

    \paragraph{Proof of Statement~4.}
    We first observe that by applying a suitable isomorphism to $\F^d$, it suffices to consider LOS inequalities for $\eta = \unitvec_1$. By renumbering blocks, we may assume $1 \in I$ (recall that $|I|$ is odd). Furthermore, by applying the $\sigma$-shift with
    \[
      \sigma(i) = \unitvec_1 \quad\text{for all }i \in I\setminus \{1\}
    \]
    and $\sigma(j) = \zerovec$ for all other $j$ we can restrict to the case $I = \{1\}$. Note that $\sigma \in \CT[d,n]$ since $|I|$ is odd.
    Hence, it remains to show that the face $F(d,n)$ of $\PIP[d,n]$ defined by
    \begin{equation}\label{eq:lifted_canonical_cut_facet_ieq}\tag{$E(d,n)$}
      \sum_{\omega \in \F^d:\omega_1=0} x^1_{\omega}+ \sum_{i=2}^n\sum_{\omega \in \F^d:\omega_1=1}x^i_{\omega} = 1
    \end{equation}
    is a facet.

    We denote by $R(d,n) \in \CT[d,n]$ the set of all cyclic transversals that correspond to vertices of $F(d,n)$.
Since we have $F(d,n) \subsetneq \PIP[d,n]$ (as, e.g., the cyclic transversal $\xi\in \CT[d,n]$ with $\xi(2)=\xi(3)=\onevec$ and $\xi(i) = \zerovec$ for all $i \in [n]\setminus\{2,3\}$ is not contained in $R(d,n)$), to complete the proof it suffices to establish the following.
\begin{nestedenvironment}{theorem}
  \begin{claim}\label{claim:facets}
    For $d \ge 1$ and $n \ge 3$ each equation $a^Tx = \beta$ with $R(d,n) \subseteq R(a,\beta)$ is a linear combination of the transversal equations for $\PIP[d,n]$ and~\eqref{eq:lifted_canonical_cut_facet_ieq}.
  \end{claim}
\end{nestedenvironment}

We prove the claim by induction on $d$, where we know that it holds for $d=1$ (see Remark~\ref{rem:CTP1n}).
Therefore, let us assume $d \ge 2$, and define for each pair $(a,b) \in \F^2$ the set
\[
  \Omega_{ab} \coloneqq \{\omega \in \F^d : \omega_1=a, \omega_d = b\}\,.
\]
A cyclic transversal $\xi \in \CT[d,n]$ is contained in $R(d,n)$ if and only if
    \begin{equation}\label{eq:odd_set_facet_except}
      \xi(1) \in \Omega_{10}\cup\Omega_{11}
      \quad\text{and}\quad
      \xi(i) \in \Omega_{00}\cup\Omega_{01}
      \text{ for all }i \in \{2,\dots,n\}
    \end{equation}
    hold with the exception of exactly one of $\xi(1), \ldots, \xi(n)$.

Let $a^Tx = \beta$ be any equation with $R(d,n) \subseteq R(a,\beta)$. Denoting by $\tilde{a} \in \R^{\B(d-1,n)}$ the vector with
\[
  \tilde{a}^i_{(\omega_1,\dots,\omega_{d-1})} = a^i_{(\omega_1,\dots,\omega_{d-1},0)}
  \quad\text{for all }i \in [n] \text{ and }\omega \in \F^{d-1}
\]
we find $R(d-1,n) \subseteq R(\tilde{a},\beta)$, since for each cyclic transversal $\tilde{\xi} \in R(d-1,n)$ the cyclic transversal $\xi\in\CT[d,n]$ with $\xi(i) = (\tilde{\xi}(i)_1,\dots,\tilde{\xi}(i)_{d-1},0)$ for all $i \in [n]$ is contained in $R(d,n)\subseteq R(a,\beta)$. Thus, by the induction hypothesis and since the affine hull of $\PIP[d-1,n]$ is described by the transversal equations according to the first statement of the theorem, there are multipliers $\lambda_0\in \R$ for $E(d-1,n)$ and $\lambda_1,\dots,\lambda_n \in \R$ for the respective transversal equations for $\PIP[d-1,n]$ such that the corresponding linear combination of equations equals $\tilde{a}^T\tilde{x}=\beta$. Let us subtract from $a^Tx = \beta$ the linear combination of $E(d,n)$ and the transversal equations for $\CT[d,n]$ formed with the multipliers $\lambda_0,\lambda_1,\dots,\lambda_n$, respectively, and keep denoting the resulting equation by $a^Tx = \beta$. We have $\beta = 0$ as well as
\begin{equation}\label{eq:odd_set_facet_have}
  a^i_{\omega} = 0 \quad \text{for all } i \in [n] \text{ and }\omega \in \Omega_{00} \cup \Omega_{10} \,.
\end{equation}
To complete the proof it suffices to show
\begin{equation}\label{eq:odd_set_facet_show}
  a^i_{\omega} = 0 \quad \text{for all } i \in [n] \text{ and }\omega \in \Omega_{01} \cup \Omega_{11}\,.
\end{equation}

In order to establish~\eqref{eq:odd_set_facet_show}, let us first consider an arbitrary $\omega \in \Omega_{01}$.
For all $p,q \in [n]\setminus\{1\}$ with $p \ne q$ we observe that the cyclic transversal $\xi$ with
\[
  \xi(p) =
  \xi(q) = \omega \in \Omega_{01},\quad\text{and }
  \xi(r) = \zerovec \in \Omega_{00}
  \text{ for all }r \in [n]\setminus\{p,q\}
\]
is contained in $R(d,n)$ (with the exception in~\eqref{eq:odd_set_facet_except} at block~$1$). By~\eqref{eq:odd_set_facet_have} we conclude
\begin{equation}\label{eq:odd_set_facet:3}
  a^p_{\omega}= -a^q_{\omega}
  \quad\text{for all }p,q \in[n]\setminus\{1\} \text{ with }p\ne q\,.
\end{equation}
Furthermore, for each $i \in [n]\setminus\{1\}$ the cyclic transversal $\xi$ with
\[
  \xi(1) =
  \xi(i) = \omega \in \Omega_{01},\quad\text{and }
  \xi(j) = \zerovec \in \Omega_{00} \quad\text{for all }j \in [n]\setminus\{1,i\}
\]
is contained in $R(d,n)$ (again with the exception in~\eqref{eq:odd_set_facet_except} at block~$1$), which by~\eqref{eq:odd_set_facet_have} yields
\begin{equation}\label{eq:odd_set_facet:4}
  a^i_{\omega} = -a^1_{\omega}
  \quad\text{for all }i \in[n]\setminus\{1\}\,.
\end{equation}
From~\eqref{eq:odd_set_facet:3} and~\eqref{eq:odd_set_facet:4} we derive (using $n\ge 3$ )
\begin{equation*}
  a^i_{\omega} = 0
  \quad\text{for all }i \in[n]\,.
\end{equation*}
Thus we have shown
\begin{equation}\label{eq:odd_set_facet:5}
  a^i_{\omega} = 0 \quad \text{for all } i \in [n] \text{ and }\omega \in \Omega_{01}\,.
\end{equation}

For each $\omega \in \Omega_{11}$ we now consider the cyclic transversal $\xi$ with
\[
  \xi(1) = \omega \in \Omega_{11},\quad
  \xi(2) = \omega + \unitvec_d \in \Omega_{10},\quad
  \xi(3) = \unitvec_d \in \Omega_{01}
\]
and
\[
  \xi(i) = \zerovec \in \Omega_{00} \quad\text{for all }i \in [n]\setminus\{1,2,3\}
\]
that is contained in $R(d,n)$ (with the exception in~\eqref{eq:odd_set_facet_except} at block~$2$) in order to deduce $a^1_{\omega}=0$ via~\eqref{eq:odd_set_facet_have} and~\eqref{eq:odd_set_facet:5}. Finally, for $i \in [n]\setminus \{1\}$ we choose $j \in \{2,3\}\setminus\{i\}$ and
use the cyclic transversal $\xi$ with
\[
  \xi(i) = \omega \in \Omega_{11},\quad
  \xi(1) = \omega + \unitvec_d \in \Omega_{10},\quad
  \xi(j) = \unitvec_d \in \Omega_{01}
\]
and
\[
  \xi(j) = \zerovec \in \Omega_{00} \quad\text{for all }j \in [n]\setminus\{1,i,j\}
\]
that is contained in $R(d,n)$ (with the exception in~\eqref{eq:odd_set_facet_except} at block~$i$) in order to derive $a^i_{\omega}=0$ via~\eqref{eq:odd_set_facet_have} and~\eqref{eq:odd_set_facet:5}.
\end{proof}

     \section{Low rank block configurations}
\label{sec:lowrank}

As in Section~\ref{sec:eq_and_ieneq}, we restrict our attention to $n\ge 3$, since $\PIP[d,1]$ is a single point and $\PIP[d,2]$ is a simplex.

We start by an observation on LOS inequalities.
\begin{lemma}\label{lem:lifting_LOS}
    Let $\B$ be a block configuration in $\F^d$ and $\varphi:\F^d\rightarrow\F^{\overlinemod{d}}$ a linear map. Then the $\varphi$-lifting of any LOS inequality for $\CTP[\varphi(\blockconfig)]$ is a LOS inequality for $\CTP[\blockconfig]$ or valid for $\TP[\blockconfig]$.
\end{lemma}

\begin{proof}
    Let $\overlinemod{a}^T\overlinemod{x}\ge\beta$ be a LOS inequality for $\CTP[\varphi(\blockconfig)]$ associated with a non-zero vector $\overlinemod{\eta} \in \F^{\overlinemod{d}}\setminus\{\zerovec\}$ and an odd set $I \subseteq [n]$ (where $n$ is the length of $\blockconfig$). Then $\overlinemod{a}^T\overlinemod{x}\ge\beta$ is the $\psi$-lifting of~\eqref{eq:odd_set_ieq} with the linear form $\psi:\F^{\overlinemod{d}} \rightarrow \F$ with $\psi(\overlinemod{x}) = \overlinemod{\eta}^T\overlinemod{x}$ for all $\overlinemod{x}\in\F^{\overlinemod{d}}$. The $\varphi$-lifting $a^Tx\ge\beta$ of $\overlinemod{a}^T\overlinemod{x}\ge\beta$ is the $\psi\circ\varphi$-lifting of~\eqref{eq:odd_set_ieq}. If $\psi\circ\varphi : \F^d \rightarrow \F$ is the zero-map, then $a^Tx \ge \beta$ is the inequality $\sum_{i\in I}\sum_{\omega\in\B(i)}x^i_{\omega} \ge 1$ that is valid for $\TP[\blockconfig]$ (due to $|I|\ge 1$). Otherwise, there is some non-zero vector $\eta \in \F^d\setminus\{\zerovec\}$ with $\psi\circ\varphi\,(\omega)=\eta^T\omega$ for all $\omega \in \F^d$ and thus $a^Tx\ge\beta$ is the LOS inequality associated with $\eta$ and $I$.
\end{proof}

\begin{corollary}\label{cor:CTP_rank_1}
    For every block configuration $\blockconfig$ with $\rank(\blockconfig) = 1$ we have
    \[
      \CTP[\blockconfig] = \{x \in \TP[\blockconfig] : x \text{ satisfies all LOS inequalities}\}\,.
    \]
\end{corollary}

\begin{proof}
    We choose a linear map $\varphi : \F^2\rightarrow \F$ with $\varphi$-induced map $\Phi$ that induces an isomorphism between $\linspan(\blockconfig)$ and $\F$. According to Remark~\ref{rem:block_iso} we have
    \begin{equation}\label{eq:CTP_rank_1_1}
        \CTP[\blockconfig] = \Phi^{-1}(\CTP[\varphi(\blockconfig)])
    \end{equation}
    and
    \begin{equation}\label{eq:CTP_rank_1_2}
        \CTP[\varphi(\blockconfig)] = \pi^{\varphi(\blockconfig)}(\CTP[1,n] \cap \linsubspace[\varphi(\blockconfig)])\,.
    \end{equation}
    Remark~\ref{rem:CTP1n} shows that we have
    \[
        \CTP[1,n] = \{x \in \TP[1,n] : x \text{ satisfies all LOS inequalities}\}.
    \]
    Hence, as the $\varphi(\blockconfig)$-restrictions of LOS inequalities for $\CTP[1,n]$ are LOS inequalities for $\CTP[\varphi(\blockconfig)]$, from~\eqref{eq:CTP_rank_1_2} and Remark~\ref{rem:CTP_face_of_full_CTP} we derive
    \[
        \CTP[\varphi(\blockconfig)] = \{x \in \TP[\varphi(\blockconfig)] : x \text{ satisfies all LOS inequalities}\}\,,
    \]
    which then by~\eqref{eq:CTP_rank_1_1}, Remark~\ref{rem:lifting_ieqs}, and Lemma~\ref{lem:lifting_LOS} yields the claim.
\end{proof}

The main purpose of this section is to show that a similar result holds for cyclic transversal polytopes of rank two as well.

\begin{theorem}\label{thm:CTP_rank_2}
    For every block configuration $\blockconfig$ with $\rank(\blockconfig) = 2$ we have
    \[
      \CTP[\blockconfig] = \{x \in \TP[\blockconfig] : x \text{ satisfies all LOS inequalities}\}\,.
    \]
\end{theorem}

\begin{proof}
    By similar arguments as we detailed in the proof of Corollary~\ref{cor:CTP_rank_1} it suffices to establish
    \begin{equation}\label{eq:CTP_rankk_2_to_show}
      \CTP[2,n] = \{x \in \TP[2,n] : x \text{ satisfies all LOS inequalities}\}
    \end{equation}
    for $n \ge 2$ (as \eqref{eq:CTP_rankk_2_to_show} obviously holds for $n=1$ since $x^1_0\ge 1$ is a LOS inequality in that case).

    Let $Q(2,n)$ be the polytope at the right-hand side of~\eqref{eq:CTP_rankk_2_to_show}. We have $\CTP[2,n] \subseteq Q(2,n)$, and in order to establish the reverse inclusion it suffices to show that every (linear) $\ge$-inequality that is valid for $\CTP[2,n]$ is also valid for $Q(2,n)$. Thus
    we only need to exhibit, for each $c \in \R^{\B(2,n)}$, a cyclic transversal $\xi^{\star} \in \CT[2,n]$ whose incidence vector is an optimal solution to $\min\{c^Tx : x \in Q(2,n)\}$. Towards this end, we choose $\xi^{\star}$ as a cyclic transversal with minimal $c$-weight. Applying the $\xi^{\star}$-shift we may assume that the zero-transversal is $c$-minimal. Furthermore, since adding multiples of coefficient vectors of transversal equations to $c$ does not change the set of optimal solutions of $\min\{c^Tx : x \in Q(2,n)\}$ we may assume $c^i_{00} = 0$ for all $i \in [n]$ (where we denote the elements of $\F^2$ by $00$, $10$, $01$, and $11$ in the obvious way).

    With $\Omega \coloneqq \{10,01,11\}$
    the $c$-minimality of the zero-transversal implies
    \begin{equation}\label{eq:C2n:ij}
        c^i_{\omega} + c^j_{\omega} \ge 0 \quad\text{for all }i,j \in [n], i\ne j, \omega\in\Omega
    \end{equation}
    (due to $\omega + \omega = 00$) and
    \begin{equation}\label{eq:C2n:ijk}
        c^i_{10} + c^j_{01} + c^k_{11} \ge 0 \quad\text{for all }i,j,k \in [n], i\ne j, j\ne k, k\ne i
    \end{equation}
    (due to $10+01+11=00$). It follows readily from \eqref{eq:C2n:ij} and \eqref{eq:C2n:ijk} that there are at most two blocks with negative $c$-coefficients. By possibly renumbering the blocks we thus may assume
    \begin{equation*}c^i_{\omega} \ge 0 \quad\text{for all }i \ge 3, \omega \in \Omega\,.
    \end{equation*}

    The nonnegativity inequalities $x^i_{\omega} \ge 0$ with $i \in [n]$ and $\omega \in \Omega$ are satisfied at equality by the incidence vector of the zero-transversal $(00,\dots,00)$, and so are the inequalities
    \begin{align*}
        -x^i_{10} - x^{i}_{01} + \sum_{j \in [n]\setminus \{i\}} (x^j_{10} + x^j_{01}) & \ge 0 \\
        -x^i_{01} - x^{i}_{11} + \sum_{j \in [n]\setminus \{i\}} (x^j_{01} + x^j_{11}) & \ge 0 \\
        -x^i_{11} - x^{i}_{10} + \sum_{j \in [n]\setminus \{i\}} (x^j_{11} + x^j_{10}) & \ge 0
    \end{align*}
    for all $i \in [n]$ that arise by subtracting the transversal equation $x^i_{00}+x^i_{10}+x^i_{01}+x^i_{11}=1$ from the LOS inequalities associated with $I=\{i\}$ and $\eta \in \{11,01,10\}$, respectively. We denote the coefficient vectors of those inequalities by $\los{i}{10}{01}$, $\los{i}{01}{11}$, and $\los{i}{11}{10}$, respectively. By complementary slackness it suffices to construct a linear combination $a\in\R^{\B(2,n)}$ of those coefficient vectors with nonnegative multipliers that satisfies $a^i_{\omega} \le c^i_{\omega}$ for all $i \in [n]$, $\omega \in \Omega$.
    We do so by distinguishing two cases.

    First we consider the case that there is at most one block in which $c$ has negative components, thus, after possibly swapping the first and the second block, we have
    \begin{equation}\label{eq:C2n:case1}
        c^i_{\omega} \ge 0 \quad\text{for all }i\ge 2, \omega\in \Omega\,.
    \end{equation}
    In this case, let us observe that the system
    \[
        \begin{array}{rcrcrclcll}
            \max\{-c^1_{10},0\} & \le & \lambda & & & + & \nu & \le & c^i_{10} & (i\ge 2)\\
            \max\{-c^1_{01},0\} & \le & \lambda & + & \mu & & & \le & c^i_{01} & (i\ge 2)\\
            \max\{-c^1_{11},0\} & \le & & & \mu & + & \nu & \le & c^i_{11} & (i\ge 2)
        \end{array}
    \]
    of linear inequalities in variables $\lambda,\mu,\nu \in \R$ is feasible (due to~\eqref{eq:C2n:ij} and~\eqref{eq:C2n:case1} as well as the non-singularity of the coefficient matrix). Among its solutions, let
    $(\lambda^{\star},\mu^{\star},\nu^{\star})$ be one which maximizes $\min\{\lambda,\mu,\nu\}$. Since the coefficient matrix is invariant under simultaneous permutations of rows and columns, we may assume $\lambda^{\star}=\min\{\lambda^{\star},\mu^{\star},\nu^{\star}\}$ after possibly applying an automorphism of $\F^2$.
    As each solution of the system has at most one negative component, we conclude $\mu^{\star},\nu^{\star}\ge 0$.

    Let us consider the case $\lambda^{\star}<0$ in more detail, which in particular implies $\lambda^{\star} < \mu^{\star},\nu^{\star}$.
    Due to the maximality property of $(\lambda^{\star},\mu^{\star},\nu^{\star})$ we then have
    \begin{equation}\label{eq:C2n:negative_lambda_star}
        \mu^{\star} + \nu^{\star} = \max\{-c^1_{11},0\}
        \quad\text{and}\quad
        \lambda^{\star} + \nu^{\star} = c^{i_{10}}_{10}, \quad
        \lambda^{\star} + \mu^{\star} = c^{i_{01}}_{01}
    \end{equation}
    for some $i_{10},i_{01} \ge 2$ (as otherwise there is some $\varepsilon>0$ such that
    $(\lambda^{\star}+\varepsilon,\mu^{\star}-\varepsilon,\nu^{\star}-\varepsilon)$ or
    $(\lambda^{\star}+\varepsilon,\mu^{\star}-\varepsilon,\nu^{\star}+\varepsilon)$ or $(\lambda^{\star}+\varepsilon,\mu^{\star}+\varepsilon,\nu^{\star}-\varepsilon)$, respectively, are feasible). Since $\lambda^{\star} < 0$ implies $\mu^{\star}+\nu^{\star} > 0$ (due to $\lambda^{\star}+\mu^{\star}\ge 0$ and $\lambda^{\star}+\nu^{\star}\ge 0$) we deduce $\mu^{\star}+\nu^{\star} = -c^1_{11}$ from the first statement in~\eqref{eq:C2n:negative_lambda_star}. By subtracting the first equation in~\eqref{eq:C2n:negative_lambda_star} from the sum of the last two ones we then obtain
    \(
        c^1_{11} + c^{i_{10}}_{10} + c^{i_{01}}_{01} = 2\lambda^{\star} < 0\,,
    \)
which by~\eqref{eq:C2n:ijk} implies $i_{10}=i_{01}$. Therefore, by possibly permuting the blocks we can assume $i_{10}=i_{01}=2$ and hence (using~\eqref{eq:C2n:negative_lambda_star} and~\eqref{eq:C2n:ijk}), for every $i \ge 3$
\begin{align*}
  -\lambda^{\star} + \nu^{\star} &= -(\lambda^{\star}+\mu^{\star}) + (\mu^{\star} + \nu^{\star})
  = -c^2_{01} - c^1_{11} \le c^i_{10} \quad\text{and}\\
  -\lambda^{\star} + \mu^{\star} &= -(\lambda^{\star}+\nu^{\star}) + (\mu^{\star} + \nu^{\star})
  = -c^2_{10} - c^1_{11} \le c^i_{01} \,.
\end{align*}

Consequently, regardless of the sign of $\lambda^{\star}$ (see the first two equations of the system) we have
\begin{equation}\label{eq:C2n:abs_lambda_star}
    |\lambda^{\star}| + \nu^{\star} \le c^i_{10}
    \quad\text{and}\quad
    |\lambda^{\star}| + \mu^{\star} \le c^i_{01}
    \quad\text{for all }i\ge 3\,.
\end{equation}

    With $\lambda^+\coloneqq\max\{\lambda^{\star},0\} \ge 0$ and $\lambda^-\coloneqq\max\{-\lambda^{\star},0\}$, thus $\lambda^+ - \lambda^- = \lambda^{\star}$ and $\lambda^+ +\lambda^- = |\lambda^{\star}|$, we now
    choose
    \begin{equation*}
      a \coloneqq \lambda^+\cdot\los{1}{10}{01} + \lambda^-\cdot\los{2}{10}{01}
        +\mu^{\star}\cdot\los{1}{01}{11}
        +\nu^{\star}\cdot\los{1}{11}{10}\,,
    \end{equation*}
    which indeed satisfies
    \begin{align*}
      a^1_{10} & = -\lambda^+ + \lambda^- - \nu^{\star} = -(\lambda^{\star}+\nu^{\star}) \le \min\{c^1_{10},0\} \le c^1_{10} \,,\\
      a^1_{01} & = -\lambda^+ + \lambda^- - \mu^{\star} = -(\lambda^{\star}+\mu^{\star}) \le \min\{c^1_{01},0\} \le c^1_{01} \,, \\
      a^1_{11} & = - \mu^{\star} - \nu^{\star} = - (\mu^{\star} + \nu^{\star}) \le \min\{c^1_{11},0\} \le c^1_{11}\,,\\
      a^2_{10} & = \lambda^+ - \lambda^- + \nu^{\star} = \lambda^{\star}+\nu^{\star} \le c^2_{10} \,,\\
      a^2_{01} & = \lambda^+ - \lambda^- + \mu^{\star} = \lambda^{\star}+\mu^{\star} \le c^2_{01} \,, \\
      a^2_{11} & = \mu^{\star} + \nu^{\star} \le c^2_{11}\,,
     \end{align*}
    as well as, for every $i \ge 3$ (recall~\eqref{eq:C2n:abs_lambda_star})
    \begin{align*}
        a^i_{10} & = \lambda^+ + \lambda^- +\nu^{\star} = |\lambda^{\star}| + \nu^{\star} \le c^i_{10} \,,\\
        a^i_{01} & = \lambda^+ + \lambda^- + \mu^{\star} = |\lambda^{\star}| + \mu^{\star} \le c^i_{01} \,,\text{ and}\\
        a^i_{11} & = \mu^{\star} + \nu^{\star} \le c^i_{11}\,.
    \end{align*}

It remains to construct a suitable linear combination in case of $c$ having negative components in more than one block. Due to~\eqref{eq:C2n:ij} and~\eqref{eq:C2n:ijk}, by possibly renumbering blocks and applying an automorphism to $\F^2$, we can assume
\begin{equation}\label{eq:C2n:case2_nonneg}
    c^i_{\omega} \ge 0 \quad\text{for all }i\ge 3, \omega\in \Omega
\end{equation}
as well as (recall~\eqref{eq:C2n:ij})
\begin{equation}\label{eq:C2n:case2_others}
    c^1_{10} < 0, \quad c^1_{10} \le c^1_{11}
    \quad\text{and}\quad
    c^2_{01} < 0, \quad c^2_{01} \le c^2_{11}\,.
\end{equation}
From~\eqref{eq:C2n:ij} we furthermore deduce
\[
    0 \le c^1_{11} + c^2_{11}
    =(c^1_{10} - c^2_{01} + c^2_{11}) + (c^2_{01} - c^1_{10} + c^1_{11})\,.
\]
Therefore, by possibly swapping the first two blocks and applying the automorphism of $\F^2$ that swaps $10$ and $01$ (note that \eqref{eq:C2n:case2_others} is invariant under those two operations), we can furthermore assume
\begin{equation}\label{eq:C2n:case2_star}
    c^1_{10} - c^2_{01} + c^2_{11} \ge 0\,.
\end{equation}
We then choose
\[
    a \coloneqq \lambda\cdot\los{1}{10}{11} + \mu\cdot\los{2}{01}{11} + \nu\cdot\los{2}{10}{01}
\]
with
\[
  \nu \coloneqq \frac12 \cdot \max\{0,c^1_{10} - c^2_{01} - c^1_{11}\} \ge 0
\]
as well as
\[
  \lambda \coloneqq - c^1_{10} + \nu \ge 0
  \quad\text{and}\quad
  \mu \coloneqq -c^2_{01} - \nu \ge 0\,,
\]
where $\lambda \ge 0$ is due to $c^1_{10} < 0$ and $\mu\ge 0$ holds due to $c^2_{01} < 0$ and since by $c^1_{10} \le c^1_{11}$ (see~\eqref{eq:C2n:case2_others}) we have $\nu \le \frac12\cdot\max\{0,-c^2_{01}\}$. That choice of $a$ satisfies
\begin{align*}
    a^1_{10} & = -\lambda + \nu = c^1_{10}\,,\\
    a^1_{01} & = \mu + \nu = -c^2_{01} \stackrel{\eqref{eq:C2n:ij}}{\le} c^1_{01}\,,\\
    a^1_{11} & = -\lambda + \mu = c^1_{10} - c^2_{01} - \max\{0,c^1_{10} - c^2_{01} - c^1_{11}\}
    \le c^1_{11}\,,\\
    a^2_{10} & = \lambda -\nu = -c^1_{10} \stackrel{\eqref{eq:C2n:ij}}{\le} c^2_{10} \,,\\
    a^2_{01} & = -\mu -\nu = c^2_{01}\,,\\
    a^2_{11} & = \lambda - \mu = \max\{\underbrace{-c^1_{10} + c^2_{01}}_{\le c^2_{11} \eqref{eq:C2n:case2_star}}, \underbrace{-c^1_{11}}_{\le c^2_{11} \eqref{eq:C2n:ij}}\} \le c^2_{11}
\end{align*}
as well as, for all $i \ge 3$,
\begin{align*}
    a^i_{10} & = \lambda + \nu = \max\{\underbrace{-c^1_{10}}_{\le c^i_{10} \eqref{eq:C2n:ij}},\underbrace{- c^2_{01} - c^1_{11}}_{\le c^i_{10} \eqref{eq:C2n:ijk}}\} \le c^i_{10}\,,\\
    a^i_{01} & = \mu + \nu = -c^2_{01} \stackrel{\eqref{eq:C2n:ij}}{\le} c^i_{01}\,,\text{ and}\\
    a^i_{11} & = \lambda + \mu = -c^1_{10} - c^2_{01} \stackrel{\eqref{eq:C2n:ijk}}{\le} c^i_{11}\,,
\end{align*}
which concludes the proof.
\end{proof}

It turns out that for block configurations of rank greater than two the LOS inequalities in general are not sufficient in order to describe the corresponding cyclic transversal polytopes.

\begin{proposition}\label{prop:others_than_LOS}
    For every $d\ge 3$ and $n\ge 3$ we have
    \[
      \CTP[d,n] \subsetneq \{x \in \TP[d,n] : x \text{ satisfies all LOS inequalities}\}\,.
    \]
\end{proposition}

\begin{proof}
    With $d\ge 3$ and $n \ge 3$ we denote by $000, 100, 010, 001,111$ the vectors $\zerovec, \unitvec_1,\unitvec_2,\unitvec_3,\unitvec_1+\unitvec_2+\unitvec_3\in\F^d$, respectively, and consider the point $\overlinemod{x} \in \TP[d,n]$ with $\overlinemod{x}^1_{000} = \overlinemod{x}^2_{000} = \overlinemod{x}^3_{000} = 1/3$, $\overlinemod{x}^1_{111} = 2/3$, $\overlinemod{x}^i_{100} = \overlinemod{x}^i_{010} = \overlinemod{x}^i_{001} = \overlinemod{x}^i_{111} = 1/6$ for $i \in \{2,3\}$, $\overlinemod{x}^i_{000} = 1$ for all $i \in \{4,\dots,n\}$,
    and $\overlinemod{x}^i_{\omega} = 0$ for all other pairs $(i,\omega) \in [n]\times \F^d$.

    Then $\overlinemod{x}$ satisfies the LOS inequality~\eqref{eq:lifted_odd_set_ieq} associated with any odd set $I \subseteq [n]$ and non-zero vector $\eta \in \F^d$. Indeed, this is obviously true in case of $I \cap \{4,\dots,n\} \ne \varnothing$ due to $\overlinemod{x}^i_{000} = 1$ for all $i \in \{4,\dots,n\}$. Because of $\overlinemod{x}^1_{000} = \overlinemod{x}^2_{000} = \overlinemod{x}^3_{000} = 1/3$ it also holds for $I = [3]$. Thus, by symmetry, it remains to consider the cases with $I = \{1\}$ or $I = \{2\}$ and the projection $\eta_{[3]}$ of $\eta$ to its first three coordinates contained in $\{(0,0,0),(1,0,0),(1,1,0),(1,1,1)\}$. For $I = \{1\}$ the left-hand-side of~\eqref{eq:lifted_odd_set_ieq} then evaluates to
\begin{align*}
    \overlinemod{x}^1_{000} + \overlinemod{x}^1_{111} &= 1 \\
    \overlinemod{x}^1_{000} + \overlinemod{x}^2_{100} + \overlinemod{x}^2_{111} + \overlinemod{x}^3_{100} + \overlinemod{x}^3_{111} &= 1\,,\\
    \overlinemod{x}^1_{000} + \overlinemod{x}^1_{111} + \overlinemod{x}^2_{100} + \overlinemod{x}^2_{010} + \overlinemod{x}^3_{100} + \overlinemod{x}^3_{010} &= 5/3 \,, \text{and}\\
    \overlinemod{x}^1_{000} + \overlinemod{x}^2_{100} + \overlinemod{x}^2_{010}+ \overlinemod{x}^2_{001}+ \overlinemod{x}^2_{111}+ \overlinemod{x}^3_{100} + \overlinemod{x}^3_{010}+ \overlinemod{x}^3_{001}+ \overlinemod{x}^3_{111} &= 5/3
\end{align*}
for $\eta_{[3]}=(0,0,0), (1,0,0), (1,1,0), (1,1,1)$, respectively, and for
$I = \{2\}$ the left-hand-side of~\eqref{eq:lifted_odd_set_ieq} equals
\begin{align*}
    \overlinemod{x}^2_{000} + \overlinemod{x}^2_{100} + \overlinemod{x}^2_{010} + \overlinemod{x}^2_{001} + \overlinemod{x}^2_{111}&= 1 \\
    \overlinemod{x}^2_{000} + \overlinemod{x}^2_{010} + \overlinemod{x}^2_{001} + \overlinemod{x}^1_{111} + \overlinemod{x}^3_{100} + \overlinemod{x}^3_{111} &= 5/3 \\
    \overlinemod{x}^2_{000} + \overlinemod{x}^2_{001} + \overlinemod{x}^2_{111} + \overlinemod{x}^3_{100} + \overlinemod{x}^3_{010} &= 1 \\
    \overlinemod{x}^2_{000} + \overlinemod{x}^1_{111} + \overlinemod{x}^3_{100} + \overlinemod{x}^3_{010} + \overlinemod{x}^3_{001} + \overlinemod{x}^3_{111} &= 5/3
\end{align*}
for $\eta_{[3]}=(0,0,0), (1,0,0), (1,1,0), (1,1,1)$, respectively.

However, we have $\overlinemod{x} \not\in \CTP[d,n]$ as it violates the inequality
\begin{equation}\label{eq:sum_of_basis_ieq}
    x^1_{111} + x^2_{000} + x^2_{100} + x^2_{010} + x^2_{001} + x^3_{000} + x^3_{100} + x^3_{010} + x^3_{001} + \sum_{i=4}^n x^i_{000} \le n-1\,,
  \end{equation}
  (the left-hand-side evaluates to $7/3+n-3=n-2/3$ for $\overlinemod{x}$) which is valid for $\CTP[d,n]$, since for $\xi(1) = 111$ and $\xi(2),\xi(3) \in \{000, 100,010,001\}$ as well as $\xi(i) = 000$ for $i \in \{4,\dots,n\}$ we have $\sum_{i=1}^n \xi(i) \ne \zerovec$.
\end{proof}

We defer the investigation of generalizations of~\eqref{eq:sum_of_basis_ieq} to future work.
     \section{Extended formulations for CTPs}
\label{sec:ext_form}

Let $\blockconfig$ be a block configuration in $\F^d$ of length $n$.
We construct a flow-based extended formulations of $\CTP[\B]$ in the following rather standard way (generalizing a construction described for $\EVEN[n]$ by~\cite{CarrKonjevod2005}). Let us define a directed graph $D(\blockconfig)$ with node set
\[
  V(\blockconfig) \coloneqq \{(\zerovec,0)\} \cup \{(\sigma,i) : i \in [n-1], \sigma \in \linspan(\blockconfig)\} \cup \{(\zerovec,n)\}
\]
and arc set $A(\blockconfig) \coloneqq A^1(\blockconfig) \cup \cdots \cup A^n(\blockconfig)$ with
\[
  A^1(\blockconfig) \coloneqq \{((\zerovec,0),(\omega,1)) : \omega \in \B(1)\}\,,
\]
\[
  A^i(\blockconfig) \coloneqq \{((\sigma,i-1),(\sigma+\omega,i)) : \sigma \in \linspan(\blockconfig), \omega\in\B(i)\}\quad (i\in[n-1]\setminus \{1\})\,,
\]
and
\[
  A^n(\blockconfig) \coloneqq \{((\omega,n-1),(\zerovec,n)) : \omega \in \B(n)\}\,.
\]
Then the $(\zerovec,0)-(\zerovec,n)$-paths in $D(\blockconfig)$ are in one-two-one correspondence with the cyclic transversals of $\blockconfig$, where the path corresponding to $\xi \in \CT[\blockconfig]$ visits node $(\sigma,i)$ if and only if $\sum_{j=1}^i \xi(j) = \sigma$ holds. The convex hull of the incidence vectors of those paths is
\begin{eqnarray}
  \flowext(\blockconfig) = \{\ y \in \R^{A(\blockconfig)} :
    \sum_{a \in \delta^{\text{out}}(\zerovec,0)} y_a & = & 1 \nonumber\\
    \sum_{a \in \delta^{\text{in}}(\sigma,i)} y_a - \sum_{a \in \delta^{\text{out}}(\sigma,i)} y_a & = & 0 \quad (i \in [n-1], \sigma \in \linspan(\blockconfig)) \nonumber \\
    y_a & \ge & 0 \quad(a \in A(\blockconfig))\ \} \label{eq:flx}\,.
\end{eqnarray}
Defining for each $i \in [n]$ and $\omega \in \B(i)$
\[
  A^i(\omega) \coloneqq \{((\sigma,i-1),(\sigma',i)) \in A^i(\blockconfig) : \sigma+ \omega = \sigma' \}
\]
it is easy to see that the linear map defined via
\[
  x^i_{\omega} = \sum_{a \in A^i(\omega)} y_a \quad\text{for all }i \in [n], \omega \in \B(i)
\]
maps $\flowext(\blockconfig)$ to $\CTP[\blockconfig]$. As the number $|A(\blockconfig)|$ of arcs clearly is bounded by $|\linspan(\blockconfig)|$ times the size of $\blockconfig$, we conclude the following.

\begin{theorem}\label{thm:flx}
  For each block configuration $\blockconfig$ in $\F^d$ there is an extended formulation of $\CTP[\blockconfig]$ of size at most $2^{\rank(\blockconfig)}\cdot\size(\blockconfig)$.
\end{theorem}
     \section{Rank relaxations}
\label{sec:rank_relax}

In this section, we are going to elaborate on a natural way to build relaxations of cyclic transversal polytopes.

Let $\B$ be a block configuration in $\F^d$.
For each linear map $\varphi : \F^d \rightarrow \F^{\overlinemod{d}}$ with $\varphi$-induced map
\(
  \Phi : \TP[\blockconfig] \rightarrow \R^{\varphi(\blockconfig)}
\)
 we have $\Phi(\CTP[\blockconfig]) \subseteq \CTP[\varphi(\blockconfig)]$ (see~\eqref{eq:Phi_of_CTP}). Therefore, the polytope
\[
  \projrelax[\blockconfig][\varphi] \coloneqq \Phi^{-1}(\CTP[\varphi(\blockconfig)]) \subseteq \TP[\blockconfig]
\]
that we refer to as the \emph{$\varphi$-relaxation} of $\CTP[\blockconfig]$ satisfies $\CTP[\blockconfig] \subseteq \projrelax[\blockconfig][\varphi]$.

Remark~\ref{rem:lifting_ieqs} readily implies the following.
\begin{remark}
  Within the setup above, if
  \[
    \CTP[\varphi(\blockconfig)] = \{\overlinemod{x} \in \TP[\varphi(\blockconfig)]:\overlinemod{A}\overlinemod{x}\ge b\}
  \]
  holds and $Ax\ge b$ is the system of inequalities obtained by $\varphi$-lifting the inequalities in $\overlinemod{A}\overlinemod{x}\ge b$, then we have
  \[
    \projrelax[\blockconfig][\varphi] = \{x\in\TP[\blockconfig]:Ax\ge b\}\,.
  \]
\end{remark}

\begin{example}\label{ex:R1}
  For each block configuration $\blockconfig$ in $\F^d$ and every $\eta \in \F^d\setminus\{\zerovec\}$ defining the linear form $\varphi : \F^d \rightarrow \F$ with $\varphi(\omega) = \eta^T\omega$ for all $\omega \in \F^d$ we have
  \[
    \projrelax[\blockconfig][\varphi] = \{x \in \TP[\blockconfig]: x \text{ satisfies all LOS inequalities associated with }\varphi\}\,.
  \]
\end{example}

\begin{lemma}\label{lem:RP_Phi_psi}
  If $\blockconfig$ is a block configuration in $\F^d$ and
  $\varphi: \F^d \rightarrow \F^{\overlinemod{d}}$ as well as $\psi:\F^{\overlinemod{d}}\rightarrow\F^{d'}$ are linear maps, then we have
  \begin{equation}\label{eq:RP_phi_psi_eq}
    \projrelax[\blockconfig][\psi\circ\varphi] = \Phi^{-1}( \projrelax[\varphi(\blockconfig)][\psi])
  \end{equation}
  (with the $\varphi$-induced map $\Phi : \TP[\blockconfig] \rightarrow \R^{\varphi(\blockconfig)}$) and
  \begin{equation}\label{eq:RP_phi_psi}
    \projrelax[\blockconfig][\varphi] \subseteq \projrelax[\blockconfig][\psi\circ\varphi]\,.
  \end{equation}
  If $\psi$ is injective (on the subspace spanned by the blocks in $\varphi(\blockconfig)$) then we have equality in~\eqref{eq:RP_phi_psi}.
\end{lemma}

\begin{proof}
With the $\psi$-induced map $\Psi : \TP[\varphi(\blockconfig)] \rightarrow \R^{\psi(\varphi(\blockconfig))}$ the $(\psi\circ\varphi)$-induced map $\TP[\blockconfig] \rightarrow \R^{\psi(\varphi(\blockconfig))}$ is $\Psi\circ\Phi$, which yields
\begin{equation}\label{eq:RP_phi_psi_proof_1}
  \projrelax[\blockconfig][\psi\circ\varphi] =
  \Phi^{-1}(\Psi^{-1}(\CTP[\psi(\varphi(\blockconfig))])) =
  \Phi^{-1}( \projrelax[\varphi(\blockconfig)][\psi])\,,
\end{equation}
thus~\eqref{eq:RP_phi_psi_eq}.
Moreover, combining the first equation in~\eqref{eq:RP_phi_psi_proof_1} with
\begin{equation}\label{eq:RP_phi_psi_proof_2}
  \Psi(\CTP[\varphi(\blockconfig)]) \subseteq \CTP[\psi(\varphi(\blockconfig))]
\end{equation}
  we obtain
\begin{equation}
\begin{aligned}\label{eq:RP_phi_psi_proof_3}
  \projrelax[\blockconfig][\psi\circ\varphi]
  &= \Phi^{-1}(\Psi^{-1}(\CTP[\psi(\varphi(\blockconfig))]))\\
  &\supseteq \Phi^{-1}(\Psi^{-1}(\Psi(\CTP[\varphi(\blockconfig)]))) \\
  &\supseteq \Phi^{-1}(\CTP[\varphi(\blockconfig)]) \\
  &= \projrelax[\blockconfig][\varphi]\,,
\end{aligned}
\end{equation}
thus~\eqref{eq:RP_phi_psi}.

If $\psi$ is injective then $\Psi$ in fact induces an isomorphism between $\CTP[\varphi(\blockconfig)]$ and $\CTP[\psi(\varphi(\blockconfig))]$, hence we have equality in~\eqref{eq:RP_phi_psi_proof_2}, and hence in~\eqref{eq:RP_phi_psi_proof_3}.
\end{proof}

\begin{definition} Let the dimension of the image of a linear map $\varphi$ be denoted by $\rank(\varphi)$.
For each block configuration $\blockconfig$ in $\F^d$ and for every $r \in \{1,2,\dots\}$, we call
\[
  \projrelax[\blockconfig][r] \coloneqq \bigcap \{\projrelax[\blockconfig][\varphi]:\varphi\text{ linear map of $\F^d$ with } \rank(\varphi)\le r\}
\]
the \emph{rank-$r$ relaxation} of $\CTP[\blockconfig]$.
\end{definition}

\begin{remark}
  For each block configuration $\blockconfig$ of rank $R$ we have
  \[
    \TP[\blockconfig] \supseteq \projrelax[\blockconfig][1] \supseteq \projrelax[\blockconfig][2] \supseteq \cdots \supseteq \projrelax[\blockconfig][R] = \CTP[\blockconfig]\,.
  \]
\end{remark}

\begin{proposition}\label{prop:RPr}
  For each block configuration $\blockconfig$ in $\F^d$ and $r \in [d]$ we have
  \begin{equation}\label{eq:RPr}
    \projrelax[\blockconfig][r] = \bigcap \{\projrelax[\blockconfig][\varphi]:\varphi : \F^d \rightarrow \F^r \text{ linear map with } \rank(\varphi)=r\}\,.
  \end{equation}
  In particular, $\projrelax[\blockconfig][r]$ is a polytope.
\end{proposition}

\begin{proof}
  The inclusion $\subseteq$ is clear. In order to establish the reverse inclusion, let $\varphi : \F^d \rightarrow \F^{\overlinemod{d}}$ be any linear map with $\rank(\varphi)\le r$. It suffices to show that there is a linear map $\varphi' : \F^d \rightarrow \F^r$ with $\rank(\varphi') = r$ and $\projrelax[\blockconfig][\varphi'] \subseteq \projrelax[\blockconfig][\varphi]$.

  Let $U\subseteq \F^{\overlinemod{d}}$ be any subspace that contains the image of $\varphi$ and has dimension $r$ (recall that we have $\rank(\varphi)\le r$), and choose any isomorphism $\psi : \F^r \rightarrow U$. Then $\varphi' \coloneqq \psi^{-1}\circ\varphi : \F^d \rightarrow \F^r$ has rank $r$ and satisfies $\varphi = \psi\circ\varphi'$. According to Lemma~\ref{lem:RP_Phi_psi} we have
  \(
    \projrelax[\blockconfig][\varphi'] \subseteq \projrelax[\blockconfig][\psi\circ\varphi'] = \projrelax[\blockconfig][\varphi]
  \)
  as desired.

  As there are only finitely many maps $\F^d \rightarrow \F^r$, and since each $\projrelax[\blockconfig][\varphi]$ is a polytope, \eqref{eq:RPr} immediately implies that $\projrelax[\blockconfig][r]$ is a polytope as well.
\end{proof}

We stated in Remark~\ref{rem:CTP_face_of_full_CTP} that every cyclic transversal polytope is isomorphic to a face of a full cyclic transversal polytope in a natural way. A corresponding relationship can easily be seen to hold for rank relaxations as well.

\begin{remark}
  Let $\blockconfig$ be a block configuration of length $n$ in $\F^d$. We have
  \[
    \projrelax[\blockconfig][\varphi] = \pi^{\B}(\projrelax[d,n][\varphi] \cap \linsubspace[\blockconfig])
  \]
  for all linear maps $\varphi:\F^d\rightarrow \F^{\overlinemod{d}}$ and
  \[
    \projrelax[\blockconfig][r] = \pi^{\B}(\projrelax[d,n][r] \cap \linsubspace[\blockconfig])
  \]
  for all $r\in \{1,2,\dots\}$. \end{remark}

Proposition~\ref{prop:RPr} and Example~\ref{ex:R1} imply the following.
\begin{corollary}\label{cor:R1}
  For each block configuration we have
  \[
    \projrelax[\blockconfig][1] = \{x \in \TP[\blockconfig]: x \text{ satisfies all LOS inequalities}\}\,.
  \]
\end{corollary}

The fact that cyclic transversal polytopes of rank two are completely described (as subpolytopes of the corresponding transversal polytopes) by the LOS inequalities turns out to imply the following.

\begin{theorem}\label{thm:R2}
  For each block configuration $\blockconfig$ we have
  \[
    \projrelax[\blockconfig][2] = \projrelax[\blockconfig][1]\,.
  \]
\end{theorem}

\begin{proof}
It suffices to establish $\projrelax[\blockconfig][1] \subseteq \projrelax[\blockconfig][2]$ for each block configuration $\blockconfig$ in $\F^d$, i.e., we need to show
\begin{equation}\label{eq:R2_1}
  \projrelax[\blockconfig][1] \subseteq \Phi^{-1}(\CTP[\varphi(\blockconfig)])
\end{equation}
for each linear map $\varphi : \F^d \rightarrow \F^{\overlinemod{d}}$ with $\rank(\varphi) = 2$ and $\varphi$-induced map $\Phi$. However, this holds true because due to $\rank(\varphi(\blockconfig)) \le 2$ we have
\[
  \CTP[\varphi(\blockconfig)] = \{x \in \TP[\varphi(\blockconfig)] : x \text{ satisfies all LOS inequalities}\}
\]
(see Theorem~\ref{thm:CTP_rank_2}), hence due to Remark~\ref{rem:lifting_ieqs} the right-hand-side of~\eqref{eq:R2_1} is the set of all $x \in \TP[\blockconfig]$ that satisfy the $\varphi$-liftings of the LOS inequalities for $\CTP[\varphi(\blockconfig)]$, which according to Lemma~\ref{lem:lifting_LOS} are LOS inequalities for $\CTP[\blockconfig]$, and thus are valid for $\projrelax[\blockconfig][1]$ (see Corollary~\ref{cor:R1}).
\end{proof}

From Theorem~\ref{thm:R2}, Corollary~\ref{cor:R1}, and Proposition~\ref{prop:others_than_LOS} we conclude the following.
\begin{remark}
  For all $n\ge 3$ we have $\projrelax[\B(d,n)][3] \subsetneq \projrelax[\B(d,n)][2]$.
\end{remark}

One can show~\cite{KaibelTrappe2024} that for each graph $G$ the non-trivial facets of the cyclic transversal polytopes $\CTP[\B^{\match(G)}]$ (which is isomorphic to the matching polytope of~$G$, see Corollary~\ref{cor:matching}) are defined by LOS inequalities. Thus, Example~\ref{ex:R1} implies
\(
  \projrelax[\CTP[\B^{\match(G)}]][1] = \CTP[\B^{\match(G)}]
\).
From Rothvoss' seminal work~\cite{Rothvoss2017} we know that in general the extension complexity of the matching polytope of a graph $G$ cannot be bounded polynomially in the size of $G$.
Hence we conclude the following.

\begin{remark}\label{rem:RP_ext_compl}
  The extension complexity of $\projrelax[\blockconfig][1]$ (and of $\projrelax[\blockconfig][r]$ for any other fixed $r$) in general cannot be bounded polynomially in the size of the block configuration $\blockconfig$.
\end{remark}

However, for fixed $r$, for every block configuration $\blockconfig$ in $\F^d$ and for every single linear map $\varphi : \F^d \rightarrow \F^r$, the $\varphi$-relaxation $\projrelax[\blockconfig][\varphi]$ of $\CTP[\blockconfig]$ has extended formulations whose sizes can be bounded linearly in the size of $\blockconfig$ (see Theorem~\ref{thm:flx}). As the number of linear maps $\F^d \rightarrow \F^r$ equals $(2^r)^d = (2^d)^r$ and the size of $\B(d,n)$ equals $n2^d \ge 2^d$, defining
\[
  \projrelax[d,n][\varphi] \coloneqq \projrelax[\B(d,n)][\varphi]
  \quad\text{and}\quad
  \projrelax[d,n][r] \coloneqq \projrelax[\B(d,n)][r]
\]
we conclude the following by combining extended formulations for
\(
  \projrelax[d,n][\varphi]
\)
for all linear maps $\varphi:\F^d\rightarrow\F^r$ (with $\rank(\varphi)=r$).

\begin{remark}
  For each fixed $r$, for every $d$ and $n$ the relaxation $\projrelax[d,n][r]$ of the full cyclic transversal polytope $\CTP[d,n]$ has an extended formulation whose size is bounded linearly in the size $n2^d$ of $\B(d,n)$.
\end{remark}

Combinations of extended formulations of several $\varphi$-relaxations as considered above in fact may be worth to be investigated further, e.g. in the following direction. Suppose that $\blockconfig$ is a block configuration in $\F^d$, $\varphi_1,\dots,\varphi_m : \F^d\rightarrow \F^r$ are linear maps, and for each $k \in [m]$
\[
  \projrelax[\blockconfig][\varphi_k] = \{x \in \TP[\blockconfig] : x = C_ky_k, A_ky_k \ge b_k, y_k \in \R^{q_k}\}
\]
(with $C_k \in \R^{\B \times q_k}$, $A_k \in \R^{p_k\times q_k}$, $b_k \in \R^{p_k}$) is an extended formulation of the $\varphi_k$-relaxation of $\CTP[\blockconfig]$. We clearly have
\begin{eqnarray}
  \projrelax[\blockconfig][r]
  & \subseteq &
    \projrelax[\blockconfig][\varphi_1] \cap \cdots \cap \projrelax[\blockconfig][\varphi_m] \nonumber\\
  & = &
  \{x \in \TP[\blockconfig] : x = C_ky_k, A_ky_k \ge b_k, y_k \in \R^{q_k}, k \in [m]\} \label{eq:intersect_ext_forms}\\
  & \eqqcolon &
    \projrelax[\blockconfig][\varphi_1,\dots,\varphi_m]\,.\nonumber
\end{eqnarray}
Let us choose, for each incidence vector $x(\xi)$ of a cyclic transversal $\xi \in \CT[\blockconfig]$ and for every $k \in [m]$, a specific preimage $y_k(\xi) \in \R^{q_k}$ with $x(\xi) = C_ky_k(\xi)$ and $A_ky_k(\xi)\ge b_k$. We now can add any linear inequalities that are valid for
\[
  \{(y_1(\xi),\dots,y_m(\xi)) : \xi \in \CT[\blockconfig]\}
\]
to the system in~\eqref{eq:intersect_ext_forms} and thus obtain a potentially stronger relaxation $\projrelaxstrengthened[\blockconfig][\varphi_1,\dots,\varphi_m]$ with
\[
  \CTP[\blockconfig] \subseteq \projrelaxstrengthened[\blockconfig][\varphi_1,\dots,\varphi_m] \subseteq \projrelax[\blockconfig][\varphi_1,\dots,\varphi_m]\,.
\]

We conclude by remarking that the concept of rank-$r$ relaxations invites to introduce the following specific notion of rank of the facets of cyclic transversal polytopes.
If $\blockconfig$ is a block configuration and
$a^Tx \ge \beta$ is an inequality that is valid for $\CTP[\blockconfig]$, then we define the \emph{CT-rank} of $a^Tx \ge \beta$ to be the smallest value $r \in \{0,1,2,\dots\}$ for which $a^T x \ge \beta$ is valid for $\projrelax[\blockconfig][r]$ (with $\projrelax[\blockconfig][0] \coloneqq \TP[\blockconfig]$).

\begin{corollary}
  For each block configuration $\blockconfig$ a facet-defining inequality for $\CTP[\blockconfig]$ has CT-rank one if and only if the facet it defines is defined by a LOS inequality. No valid inequality for $\CTP[\blockconfig]$ has CT-rank two.
\end{corollary}

     \section{Conclusions}

We hope to have demonstrated that (binary) cyclic transversal polytopes form a rather general class of polytopes that comprises prominent polytopes (such as matching and, more generally, stable set, and even 2-SAT polytopes as well as cut polytopes) while being well-structured enough to allow to come up with strong enough results on the whole class (e.g., the LOS inequalities) to make them interesting.

There are several directions into which it seems worth to continue the studies of cyclic transversal polytopes. We mention a few of them.

One interesting aspect is to understand better which concrete polytopes are affinely isomorphic to cyclic transversal polytopes and which are not. We saw that the traveling salesman polytope for five cities is not one of them (Example~\ref{ex:tsp}). In fact, we conjecture (see also Section~\ref{sec:expressive_power}) that traveling salesman polytopes for more than five cities in general are not affinely isomorphic to cyclic transversal polytopes.

We have not addressed the separation problem for LOS inequalities in the paper, for lack of interesting results. If we fix the vector $\eta\in\F^d$, then the separation problem becomes the separation problem over $\EVEN[n]$, thus polynomial-time solvable. Hence, the separation problem for LOS inequalities is polynomial-time solvable for fixed $d$. For variable $d$, however, the complexity status of the separation problem is unresolved.

One direction of further research clearly is the search for other classes of inequalities than LOS inequalities. Inequality~\eqref{eq:sum_of_basis_ieq} may provide a starting point for this. It is also interesting to understand which inequalities for, say, stable set, matching, or cut polytopes correspond to the LOS inequalities for the cyclic transversal polytopes that are affinely isomorphic to those polytopes. It turns out that at least for stable set and matching polytopes this can be analyzed quite well, leading to the well-known odd-cycle inequalities for stable set polytopes and to Edmond's inequalities for matching polytopes~\cite{KaibelTrappe2024}.

The last direction we want to mention here is to investigate further the relaxation hierarchies sketched in Section~\ref{sec:rank_relax}, in particular the strengthening possibilities indicated there. It may be worth to point out that due to the universality result presented in Theorem~\ref{thm:universality}, results on such relaxation hierarchies in principle can be applied to most polytopes that arise in combinatorial optimization.
 
    \paragraph*{Acknowledgements.}
    We thank Jannik Trappe for several valuable hints and discussions during the preparation of the manuscript.

    This work was funded by Deutsche Forschungsgemeinschaft (DFG, German Research Foundation) -- 314838170, GRK 2297 MathCoRe.

    \bibliographystyle{plainurl}

\end{document}